\documentclass[11pt,a4paper
]{article}
\usepackage{graphicx}
\usepackage{amsfonts,amsmath,amsthm,amssymb,bm}
\usepackage[T1]{fontenc}
\usepackage{float,floatflt,setspace}
\usepackage{latexsym,euscript,paralist}
\usepackage{mathenv}
\usepackage{vmargin} 
\setmarginsrb{2.4cm}{2.4cm}{2.4cm}{2.4cm}{0cm}{0.8cm}{0cm}{0.8cm}

\numberwithin{equation}{section}

\newcommand{\transposee}[1]{{\vphantom{#1}}^{\mathit t}{#1}}

\newtheorem{defi}{Definition}[section]
\newtheorem{theo}[defi]{Theorem}
\newtheorem{lem}[defi]{Lemma}

\newtheorem{prop}[defi]{Proposition}
\newtheorem{cor}[defi]{Corollary}
\newtheorem{rem}[defi]{Remark}

\let\lm\limits
\let\ds\displaystyle

\def\ve{{\varepsilon}}
\def\pa{{\partial}}

\def\R{{\mathbb{R}}}
\def\nit{\noindent}
\def\di{\mbox{d}}
\def\Om{\Omega}
\def\om{\omega}
\def\intO{\int\lm_\Om}
\def\intOt{\int\lm_{\Om_T}}

\def\ve{\varepsilon}
\def\vphi{\varphi}
\def\lt{\left}
\def\rt{\right}

\def\C{\mathcal C}
\def\a{\tau_{11}}
\def\b{\tau_{12}}
\def\c{\tau_{22}}
\def\ee{{\ve\eta}}
\def\bu{\boldsymbol u}
\def\bv{\boldsymbol v}
\def\bs{\boldsymbol \sigma}
\def\bt{\boldsymbol \tau}
\def\bn{\boldsymbol n}
\def\L{\boldsymbol L}
\def\H{\boldsymbol H}
\def\bds{\boldsymbol}

\usepackage[multiple]{footmisc}
\DefineFNsymbols{numbers}[text]{1 2 3 *}
\setfnsymbol{numbers}
\renewcommand{\multfootsep}{\, ,}

\begin{document}
\onehalfspace

\title{Viscoelastic fluids in thin domains: a mathematical proof}

\author{{\bf Guy Bayada\footnote{INSA-Lyon - Institut Camille Jordan - CNRS UMR 5208}\footnote{INSA-Lyon - LAMCOS - CNRS UMR 5259}, Laurent Chupin\footnotemark[2] \,and B\'er\'enice Grec\footnote{Ecole Centrale de Lyon - Institut Camille Jordan - CNRS UMR 5208}\footnote{E-mail: berenice.grec@insa-lyon.fr, Tel: +33 4 72 43 70 40, Fax: +33 4 72 43 85 29.}}\\
Batiment L\'eonard de Vinci - 21, avenue Jean Capelle\\
69 621 Villeurbanne cedex - France
}

\date{}

\maketitle
\begin{abstract}
The present paper deals with non Newtonian viscoelastic flows of Oldroyd-B type in thin domains. Such geometries arise for example in the context of lubrication. More precisely, we justify rigorously the asymptotic model obtained heuristically by proving the mathematical convergence of the Navier-Stokes/Oldroyd-B sytem towards the asymptotic model.
\end{abstract}


\section{Introduction}
This paper concerns the study of a viscoelastic fluid flow in a thin gap, the motion of which is imposed due to non homogeneous boundary conditions.

When a Newtonian flow is contained between two close given surfaces in relative motion, it is well known that it is possible
to replace the Stokes or Navier-Stokes equations governing the fluid's motion by a simpler asymptotic model. The asymptotic pressure is proved to be independent of the normal direction to the close surfaces and obeys the Reynolds thin film equation whose coefficients include the velocities, the geometrical description of the surrounding surfaces and some rheological characteristics of the fluid. As a following step, the computation of this pressure allows an asymptotic velocity of the fluid to be easily computed. Such asymptotic procedure first proposed in a formal way by Reynolds \cite{BC86} has been rigorously confirmed for Newtonian stationary flow \cite{ABC94}, and then generalized in a lot of situations covering numerous applications for both compressible fluid \cite{MS05}, unsteady cases \cite{BCC99}, multifluid flows \cite{Pao03}.

 It is well known however that in numerous applications, the fluid to be considered is a non Newtonian one. This is the case
for numerous biological fluids, modern lubricants in engineering applications due to the additives they contain, polymers in injection or molding process. In all of these applications, there are situations in which the flow is anisotropic. It is usual to take account of this geometrical effect in order to simplify the three-dimensional equations of the motion, trying to recover two dimensional Reynolds like equation with respect to the pressure only. Such procedures are more often heuristic ones. Nevertheless, some mathematical works appeared in the literature to justify them.  They include thin film asymptotic studies of Bingham flow \cite{BK04}, quasi Newtonian flow (Carreau's law, power law or Williamson's law, in which various stress-velocity relations are chosen: \cite{BMT93}, \cite{BT95}, \cite{ST05}) and also micro polar ones \cite{BL96}. It has been possible to obtain rigorously some thin film approximation for such fluids using a so called generalized Reynolds equation for the pressure.

However in the preceding examples, elasticity effects are neglected. Introduction of such viscoelastic behavior is characterized by the Deborah number which is related to the relaxation time. One of the most popular laws is the Oldroyd-B model whose constitutive equation is an interpolation between purely viscous and purely elastic models, thus introducing an additional parameter which describes the relative  proportion of both behaviors. A formal procedure has been proposed in \cite{BCM07}. However, the asymptotic system so obtained lacks the usual characteristic of classical generalized Reynolds equation as it has not been possible to gain an equation in the asymptotic pressure only. Both velocity $u^*$ and pressure $p^*$ are coupled by a non linear system.

It is the goal of this paper to justify rigorously this asymptotic system. Section \ref{beginning} is devoted to the precise statement of the 3-D problem. One difficulty has been to find an existence theorem for the general Oldroyd-B model, acting as a starting point for the mathematical procedure. Most of the existence theorems, however, deal with small data or small time assumptions. To control this kind of property with respect to the smallness of the gap appears somewhat difficult. So we are led to consider a more particular Oldroyd-B model, for which unconditional existence theorem has been proved \cite{LM00}.  Moreover, a specific attention is devoted to the boundary conditions to be introduced both on the velocity and on the stress. The goal is to use "well prepared" boundary conditions so as to prevent boundary layer on the lateral side of the domain.

In Section \ref{asymptexp}, after suitable scaling procedure, asymptotic expansions of both pressure, viscosity and stress are introduced, taking into account the previous formal results from \cite{BCM07}. Section \ref{ssec:limit} is mainly concerned with the proof of some additional regularity properties for the formal asymptotic solution. Assuming some restrictions on the rheological parameters, it will be proved that it is possible to gain a $C^k$ regularity for $p^*$ , $k>1$, which in turn improves the regularity of $u^*$ and the stress tensor $\sigma^*$. This result is obtained by introducing a differential Cauchy system satisfied by the derivative of $p^*$. Finally, section \ref{remainders}, is devoted to the convergence towards zero of the second term of the asymptotic expansions, which in turn proves the convergence of the solution of the real 3-D problem towards $u^*$, $p^*$, $\sigma^*$ (Theorems \ref{th:cv} and \ref{mainth_p}).

\section{Introduction of the problem and known results}\label{beginning}
\subsection{Formulation of the problem}
We consider unsteady incompressible flows of viscoelastic fluids, which are ruled by Oldroyd's law, in a thin domain $\hat\Om^\ve=\{(x,y)\in \R^n, x\in \om ~ \text{and} ~ 0<y<\ve h(x)\}$, where $\om$ is an $(n-1)$-dimensional domain, with $n=2$ or $n=3$ ($x=x_1$ or $x=(x_1,x_2)$), as in Figure \ref{fig:domain}.  
\begin{figure}[H]
\centering
\includegraphics[trim = 1cm 3.5cm 1cm 1cm, clip, width=10cm]{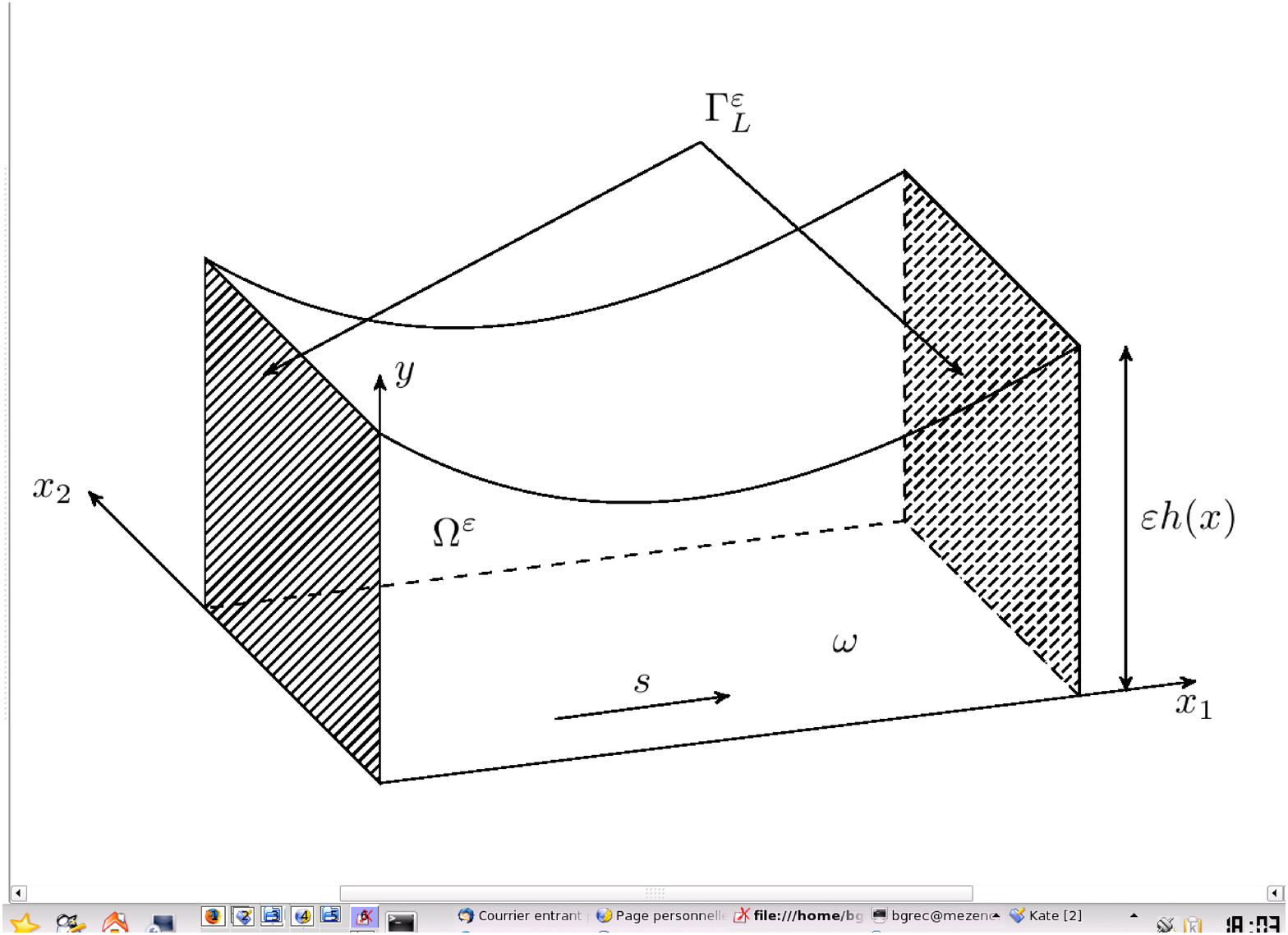}
\caption{Domain $\hat\Omega_\varepsilon$ \label{fig:domain}}
\end{figure}

\nit The following hypotheses on $h$ are required: 
\begin{equation*}
 \forall x\in \om, ~0 <h_0\leq h(x)\leq h^\ve_M, \quad \text{and} \quad h^\ve\in \mathcal C^1(\bar\om).
\end{equation*}
Let $\hat \bu^\ve=(\hat u^\ve_1,\hat u^\ve_2,\hat u^\ve_3)$ be the velocity field in the three-dimensional case, or $\hat \bu^\ve=(\hat u^\ve_1, \hat u^\ve_2)$ in the two-dimensional case, $\hat p^\ve$ the pressure, and $\hat \bs^\ve$ the stress symmetric tensor in the domain $\hat\Om^\ve$. Bold letters stand for vectorial or tensorial functions, the notation $\hat f$ corresponds to a function $f$ defined in the domain $\hat\Om^\ve$, and the superscript $^\ve$ denotes the dependence on $\ve$.

\paragraph{Formulation of the problem} 
The following formulation of the problem holds in $(0,\infty)\times \hat\Om^\ve$:
\begin{equation}
\{\begin{array}{rcl}\label{EQve}
\rho ~\pa_t \hat\bu^\ve + \rho ~\hat\bu^\ve \cdot \nabla \hat\bu^\ve - (1-r) \nu\,\Delta\hat \bu^\ve + \nabla \hat p^\ve &=& \nabla \cdot \hat\bs^\ve \,,\\
\nabla \cdot \hat\bu^\ve &=& 0 \,,\\
\lambda ~\lt(\pa_t \hat\bs^\ve + \hat\bu^\ve \cdot \nabla\hat \bs^\ve + g(\hat\bs^\ve,\nabla \hat\bu^\ve)\rt) + \hat\bs^\ve &=& 2r \nu D(\hat\bu^\ve)\,,
\end{array} \rt.
\end{equation}
where the nonlinear terms $g(\hat\bs^\ve,\nabla \hat\bu^\ve)$, the vorticity tensor $W(\hat\bu^\ve)$ and the deformation tensor $D(\hat\bu^\ve)$ are given by: 
\begin{gather*}
 g(\hat\bs^\ve,\nabla \hat\bu^\ve)= - W(\hat\bu^\ve) \cdot \hat\bs^\ve + \hat\bs^\ve \cdot W(\hat\bu^\ve),\\
 W(\hat\bu^\ve) = \dfrac{\nabla\hat \bu^\ve - \transposee{\nabla \hat\bu^\ve}}2\quad \text{and}  \quad D(\hat\bu^\ve) = \dfrac{\nabla \hat\bu^\ve + \transposee{\nabla \hat\bu^\ve}}2\,.
\end{gather*}

\nit In this formulation, the physical parameters are the viscosity $\nu$, the density $\rho$, and the relaxation time  $\lambda$. The parameter $\lambda$ is related to the viscoelastic behavior and the Deborah number. The parameter $r\in [0,1)$ describes the relative proportion of the viscous and elastic behavior. 

\paragraph{Initial conditions} 
This problem is considered with the following initial conditions:
\begin{equation}\label{IC}
\hat\bu^\ve|_{t=0} = \hat\bu^\ve_0, \quad \hat\bs^\ve|_{t=0} = \hat\bs^\ve_0,
\end{equation}
for $\hat\bu^\ve_0 \in \L^2(\hat\Om^\ve)$, $\hat\bs^\ve_0\in \L^2(\hat\Om^\ve)$. The bold notation $\L^2(\hat\Om^\ve)$ denotes the set of vectorial or tensorial functions whose all components belong to $L^2(\hat\Om^\ve)$.

\paragraph{Boundary conditions} Dirichlet boundary conditions are set on top and bottom of the domain, and the conditions on the lateral part of the boundary $\hat\Gamma^\ve_L$, defined by
\begin{equation*}
 \hat\Gamma^\ve_L=\{(x,y)\in \R^n, x\in \pa\om ~ \text{and} ~0 <y< \ve h(x)\},
\end{equation*}
will be specified later (in section \ref{ssec:BC}). Therefore, it is possible to write the boundary conditions in a shortened way:
\begin{equation}\label{BC}
\hat\bu^\ve|_{\pa\hat\Om^\ve} = \hat{\bds J}^\ve,
\end{equation}
where $\hat{\bds J}^\ve$ is a given function such that $\hat{\bds J}^\ve\in \H^{1/2}(\pa\hat\Om^\ve)$ and satisfying $\hat{\bds J}^\ve|_{y=h^\ve}=0$, $\hat{\bds J}^\ve|_{y=0}=(s,0)$. This function will be fully determined in Subsection \ref{ssec:BC}.

\nit Since $\hat\bs^\ve$ satisfies a transport equation in the domain $\hat\Om^\ve$, it remains to impose boundary conditions on $\hat\bs^\ve$ on the part of the boundary where $\hat\bu^\ve$ is an incoming velocity. Let us define $\hat\Gamma^\ve_+$ the part of $\hat\Gamma^\ve_L$ such that $\hat{\bds J}^\ve|_{\hat\Gamma_+^\ve}\cdot n < 0$, and $\hat\Gamma_-^\ve = \hat\Gamma_L^\ve\setminus \hat\Gamma_+^\ve$. We set
\begin{equation}\label{BCSigma}
\hat\bs^\ve|_{\hat\Gamma_+^\ve} = \hat{\bds\theta}^\ve,
\end{equation}
where $\hat{\bds\theta}^\ve$ is a given function in $\H^{1/2}(\hat\Gamma_+^\ve)$ which will also be determined in Subsection \ref{ssec:BC}.

\nit Moreover, since the pressure is defined up to a constant, the mean pressure is chosen to be zero: $\int\lm_{\hat\Om^\ve} \hat p^\ve=0$.

\paragraph{Notations} Let us introduce the following function space:
\begin{equation*}
 V=\{\hat{\bds\vphi} \in \H^1_0(\hat\Om^\ve),\, \nabla \cdot\hat{\bds\vphi} = 0\},
\end{equation*}
and the following notations, that will be used in the following. For $\hat f$ defined in $\hat\Om^\ve$:
\begin{compactitem}
 \item $|\hat f|$ denotes the $L^2$-norm in $\hat\Om^\ve$,
\item $|\hat f|_p$ denotes the $L^p$-norm in $\hat\Om^\ve$, for $2<p\leq +\infty$,
\item the spaces $\C^m(\overline{\hat{\Om}^\ve})$ for $m\geq 1$ are equipped with the norms $\|\hat f\|_{\C^m} = |\hat f|_\infty + \sum\lm_{i=1}^m |\hat f^{(i)}|_\infty$.
\end{compactitem}
 For  $\hat f$ defined in $\R^+ \times \hat\Om^\ve$, $\|\hat f\|_{L^\alpha(L^\beta)}$ denotes the norm of the space $L^\alpha(0,\infty,L^\beta(\hat\Om^\ve))$, with $1\leq \alpha, \,\beta \leq \infty$.

\subsection{Existence theorem in the domain $\hat\Om^\ve$}
\begin{theo}\label{th:LM}
  For $\ve>0$ fixed, problem \eqref{EQve}-\eqref{BC} admits a weak solution 
\begin{equation*}
 \hat \bu^\ve \in L^2_{loc}(0,\infty,\H^1(\hat\Om^\ve))\,, \quad \hat p^\ve\in L^2_{loc}( 0,\infty,L^2(\hat\Om^\ve))\,,\quad \hat\bs^\ve \in \C(0,\infty,\L^2(\hat\Om^\ve))\,.
\end{equation*}
\end{theo}
\begin{proof} This result is proved in \cite{LM00}. \end{proof}
\begin{rem}
Let us emphasize that for the following, it is essential to know the global (in time) existence of a solution for problem \eqref{EQve}-\eqref{BC}. Other existence theorems have been proved for this problem, for example in \cite{GS90}, \cite{FGO98}, \cite{Chu04}, but these theorems are either local in time (on a time interval $[0,T^\ve]$), or a small data assumption is needed. In this work, these theorems cannot be used, since there is no control on the behavior of $T^\ve$ (or equivalently of the data) when $\ve$ tends to zero, in particular $T^\ve$ may tend to zero. 

\nit Consequently, this work is restricted to the specific case treated in \cite{LM00}, taking one parameter of the Oldroyd model to be zero. In all generality, the non-linear term reads $g(\bs, \nabla \bu)=- W(\bu) \cdot \bs + \bs \cdot W(\bu) -a\lt(\bs \cdot D(\bu)+  D(\bu) \cdot \bs\rt)$, which is called objective derivative. Here the parameter $a$ is taken to be zero. This case corresponds to the so-called Jaumann derivative.
\end{rem}

\begin{rem}
The following computations are made in the two-dimensional case (i.e. $\om=(0,L)$ is a one-dimensional domain) for the sake of simplicity. However, note that except for the regularity obtained for the limit problem in Section \ref{s:reg}, all estimates are independent of the dimension, thus the corresponding computations should apply to the three-dimensional case. 
\end{rem}
 
\begin{paragraph}
{Regularizing the system}
In the proof of the preceding theorem, the existence of a solution is achieved by regularization. Therefore, this study only concerns solutions obtained as the limit of a regularized problem approximating \eqref{EQve}, in which an additional term $-\eta \Delta \hat\bs^\ee$ is added to the Oldroyd equation, with $\eta >0$ a small parameter. Here a regularization of the form $-\eta \Delta (\hat \bs^\ee-\hat{\bds G})$ is chosen, with $\hat{\bds G}$ a symmetric tensor in $\H^2(\hat\Om^\ve)$ independent of $\eta$ and $\ve$ which will be precised later. After obtaining the needed energy estimates uniformly in $\eta$, we will let $\eta$ tend to zero. This approach allows to multiply the Oldroyd equation by $\hat\bs^\ee$, since $\hat\bs^\ee$ is regular enough. Of course, one can choose another regularization which leads to energy estimates which are uniform in the regularization parameter.

\nit Furthermore, because of the regularizing term, boundary conditions on the whole boundary are needed. Let us write $\hat\bs^\ee|_{\pa\hat\Om^\ve} = \hat{\bds\theta}^\ee$, where $\hat{\bds\theta}^\ee$ is now a function of $\H^{1/2}(\pa\hat\Om^\ve)$, which will be determined later by equation~\eqref{deftheta}.
\end{paragraph}

\section{Asymptotic expansions}\label{asymptexp}
\subsection{Renormalization of the domain}
Introducing a new variable $z=\dfrac{y}\ve$, the system \eqref{EQve} can be rewritten in a fixed re-scaled domain: 
\begin{equation*}
\Om=\{(x,z)\in \R^n, x\in \om ~\text{and}~ 0<z<h(x)\}.
\end{equation*}
For a function $\hat f$ defined in $\Om^\ve$, $f$ is defined in $\Om$ by $f(x,z)=\hat f(x,\ve z)$. For a function $f\in L^p(\Om)$, $|f|_p$ still denotes the $L^p$-norm in $\Om$, and similar notations hold for the other norms. Moreover, the regularizing term $\eta\Delta \bs^\ee$ is introduced. 
Denoting $\bs^\ee=\begin{pmatrix} \sigma_{11}^\ee & \sigma_{12}^\ee \\                                                                                     \sigma_{12}^\ee & \sigma_{22}^\ee \end{pmatrix}$, and similar notations for the components of $\bds G$, it holds in $(0,\infty) \times \Omega $~:
\begin{System}\label{EQ}
\rho ~\delta_t u_1^\ee
- (1-r) \nu\,\Delta_\ve u_1^\ee + \pa_x p^\ee - \pa_x \sigma_{11}^\ee - \dfrac1\ve \pa_z \sigma_{12}^\ee =0\,,\\
\rho ~\delta_t u_2^\ee
- (1-r) \nu\,\Delta_\ve u_2^\ee + \dfrac1\ve \pa_z p^\ee - \pa_x \sigma_{12}^\ee - \dfrac1\ve \pa_z \sigma_{22}^\ee=0\,,\\
\nabla_\ve\cdot \bu^\ee = 0 \,, \\
\lambda ~\lt(\delta_t \sigma_{11}^\ee
-\tilde N(\bu^\ee,\sigma_{12}^\ee)\rt) + \sigma_{11}^\ee -\eta \Delta_\ve (\sigma_{11}^\ee- G_{11})- 2r \nu \pa_x u_1^\ee=0\,, \\
\lambda ~\lt(\delta_t \sigma_{12}^\ee
+\dfrac12\tilde N(\bu^\ee,\sigma_{11}^\ee-\sigma_{22}^\ee)\rt) + \sigma_{12}^\ee -\eta \Delta_\ve (\sigma_{12}^\ee-G_{12})- r \nu \lt(\pa_x u_2^\ee+\dfrac1\ve \pa_z u_1^\ee\rt) =0\,,\\
\lambda ~\lt(\delta_t \sigma_{22}^\ee
+\tilde N(\bu^\ee,\sigma_{12}^\ee)\rt) + \sigma_{22}^\ee -\eta \Delta_\ve (\sigma_{22}^\ee-G_{22})- 2r \nu \dfrac1\ve\pa_z u_2^\ee=0\,,
\end{System}
where the convective derivative $\delta_t$ is given by $\delta_t = \pa_t + \bu^\ee\cdot \nabla_\ve$. The derivation operators are defined as follows: $\nabla_\ve =\lt(\pa_x , \dfrac1\ve\pa_z \rt)$ and $\Delta_\ve = \pa_x^2 +\dfrac1{\ve^2}\pa_z^2 $. The non-linear terms $\tilde N$ are given by $\tilde N(\bu,f)=\lt(\pa_x u_2-\dfrac1\ve \pa_z u_1\rt) f$.

\subsection{Asymptotic expansions}
It has been proposed in \cite{BCM07} that when $\eta,~\ve$ tend zero, $(\bu^\ee,p^\ee,\bs^\ee)$ tends formally to a triplet $(\bu^*,p^*,\bs^*)$  satisfying a system that will be given later in \eqref{EQ_star}. This analysis leads to the introduction of the following asymptotic expansions:
\allowdisplaybreaks\begin{gather}
 u_1^\ee=u_1^*+v_1^\ee ~\text{and} ~ u_2^\ee=\ve u_2^*+\ve v_2^\ee,
\label{dvasymptU}\\
 p^\ee=\dfrac1{\ve^2}p^*+\dfrac1{\ve^2}q^\ee,\label{dvasymptp}\\
 \bs^\ee=\dfrac1\ve\bs^*+\dfrac1\ve \bt^\ee,\label{dvasympts}
\end{gather}
with $\bs^*=\begin{pmatrix} \sigma_{11}^* & \sigma_{12}^* \\ \sigma_{12}^* & \sigma_{22}^* \end{pmatrix}$, and $\bt^\ee= \begin{pmatrix} \tau_{11}^\ee & \tau_{12}^\ee \\ \tau_{12}^\ee & \tau_{22}^\ee \end{pmatrix}$. If denoting $\bu^*=(u_1^*,u_2^*)$, and $\bv^\ee=(v_1^\ee,v_2^\ee)$, \eqref{dvasymptU} becomes $\bu^\ee = \bu^* + \bv^\ee$. 

\nit The scaling orders chosen for the pressure and the different components of the velocity field and of the stress tensor are motivated by some mathematical and physical remarks. Classically, the pressure has to be of order $1/\ve^2$ if the horizontal velocity is of order 1 (see \cite{BC86} for the rigorous explanation). On the other hand, the stress tensor has to be of order $1/\ve$ and the Deborah number $\lambda$ of order $\ve$ in order to balance the Newtonian and non-Newtonian contribution in Oldroyd equation (see \cite{BCM07}). Hence; let $\lambda=\ve \lambda^*$.

A wise choice of the function $\bds G$ in the regularizing term is $\bds G=\bs^*$. The regularity of $\bds G$ in $\H^2(\Om)$ is proved by Theorem \ref{th:reg} (where it is proved that $\pa_x^2 \bs^*\in \bds\C^0(\bar\Om)$, $\pa_x\pa_z \bs^* \in \bds\C^0(\bar\Om)$ and $\pa_z^2 \bs^* \in \bds{\C}^1(\bar\Om)$, thus $\Delta\bs^* \in \L^2(\Om)$).
A formal substitution of \eqref{dvasymptU}, \eqref{dvasymptp}, \eqref{dvasympts} in \eqref{EQ} leads to the following system:
\begin{equation}
\{\begin{aligned}\label{EQ_rest_star}
&\rho ~\di_t v_1^\ee - (1-r) \nu\,\Delta_\ve v_1^\ee + \dfrac1{\ve^2}\pa_x q^\ee -\dfrac1\ve \pa_x \tau_{11}^\ee - \dfrac1{\ve^2} \pa_z \tau_{12}^\ee= \tilde L_1^\ee+\dfrac1\ve C_1+ \dfrac1{\ve^2} C_1',\\
&\rho ~\di_t v_2^\ee - (1-r) \nu\,\Delta_\ve v_2^\ee + \dfrac1{\ve^4}\pa_z q^\ee -\dfrac1{\ve^2} \pa_x \tau_{12}^\ee - \dfrac1{\ve^3} \pa_z \tau_{22}^\ee = \dfrac1{\ve^2} \tilde L_2^\ee+\dfrac1{\ve^3} C_2+ \dfrac1{\ve^4} C_2',\\
&\nabla \cdot \bv^\ee = \nabla \cdot \bu^*,\\
&\lambda^* ~\lt(\di_t \tau_{11}^\ee -N(\bv^\ee,\tau_{12}^\ee)\rt) + \dfrac1\ve \tau_{11}^\ee -\eta \Delta_\ve \tau_{11}^\ee -2r\nu \pa_x v_1^\ee = \tilde L_{11}^\ee + \dfrac1\ve \tilde L_{11}'^\ee,\\
&\lambda^* ~\lt(\di_t \tau_{12}^\ee +\dfrac12 N(\bv^\ee,\tau_{11}^\ee-\tau_{22}^\ee)\rt) + \dfrac1\ve \tau_{12}^\ee-\eta \Delta_\ve \tau_{12}^\ee -r \nu \lt(\pa_x v_2^\ee+\dfrac1\ve \pa_z v_1^\ee\rt) = \tilde L_{12}^\ee + \dfrac1\ve \tilde L_{12}'^\ee,\\
&\lambda^* ~\lt(\di_t \tau_{22}^\ee +N(\bv^\ee,\tau_{12}^\ee)\rt) + \dfrac1\ve \tau_{22}^\ee -\eta \Delta_\ve \tau_{22}^\ee-\dfrac{2r\nu}\ve \pa_z v_2^\ee = \tilde L_{22}^\ee + \dfrac1\ve \tilde L_{22}'^\ee,
\end{aligned} \rt.
\end{equation}
with the following notations: $\di_t = \pa_t + \bv^\ee\cdot \nabla$ is the so-called convective derivative, the non-linear terms $N(\bv^\ee,f)= \lt(\ve\pa_x v_2^\ee-\dfrac1\ve \pa_z v_1^\ee\rt) f$ for $f\in L^2(\Om)$ and the following linear (with respect to $\bv^\ee$) and constant terms \label{L1C1}
\begin{align*}
\tilde L_1^\ee&= \underbrace{-\rho~\bv^\ee\cdot \nabla u_1^* -\rho~\bu^*\cdot \nabla v_1^\ee}_{\mathcal L_1^\ee} -\rho ~\pa_t u_1^* -\rho ~\bu^*\cdot \nabla u_1^* + (1-r)\nu \pa_x^2 u_1^*,\\
C_1&= \pa_x \sigma_{11}^*,\\
C_1' &= (1-r)\nu \pa_z^2 u_1^*-\pa_x p^* +\pa_z \sigma_{12}^*;\\
\tilde L_2^\ee&= \underbrace{-\rho~\ve^2\bv^\ee\cdot \nabla u_2^* - \rho~\ve^2\bu^*\cdot \nabla v_2^\ee}_{\mathcal L_2^\ee} \\
&\qquad -\rho ~\ve^2\pa_t u_2^* -\rho ~\ve^2\bu^*\cdot \nabla u_2^* + \ve^2(1-r)\nu \pa_x^2 u_2^*+(1-r)\nu \pa_z^2 u_2^* + \pa_x \sigma_{12}^*,\\
C_2&= \pa_z \sigma_{22}^*,\\
C_2'& = \pa_z p^*.
\end{align*}
For the Oldroyd equation, the following linear (with respect to $\bv$ and $\bt$) and constant terms appear:
\begin{align*}
\tilde L_{11}^\ee = &\mathcal L_{11}^\ee +\lambda^*\lt(-\pa_t \sigma_{11}^* - \bu^*\cdot \nabla \sigma_{11}^* +\ve \pa_x u_2^* \sigma_{12}^*\rt) + 2r\nu \pa_x u_1^*,\\
\text{with} \quad &\mathcal L_{11}^\ee =\lambda^*\lt(\ve \pa_x u_2^* \tau_{12}^\ee + \ve \pa_x v_2^\ee \sigma_{12}^* -\bv^\ee\cdot \nabla \sigma_{11}^* -\bu^*\cdot \nabla \tau_{11}^\ee \rt),\\
\tilde L_{11}'^\ee =& \underbrace{-\lambda^*\lt(\pa_z u_1^* \tau_{12}^\ee  + \pa_z v_1^\ee \sigma_{12}^*\rt)}_{\mathcal L_{11}'^\ee} -\lambda^* \pa_z u_1^* \sigma_{12}^* -\sigma_{11}^*;\\
\tilde L_{22}^\ee =& \mathcal L_{22}^\ee -\lambda^*\lt(\pa_t \sigma_{22}^* + \bu^*\cdot \nabla \sigma_{22}^* +\ve \pa_x u_2^* \sigma_{12}^*\rt) + 2r\nu \pa_z u_2^*,\\
\text{with} \quad &\mathcal L_{22}^\ee =-\lambda^*\lt(\ve \pa_x u_2^* \tau_{12}^\ee +\ve \pa_x v_2 \sigma_{12}^* +\bv^\ee\cdot \nabla \sigma_{22}^* +\bu^*\cdot \nabla \tau_{22}^\ee \rt),\\
\tilde L_{22}'^\ee =&\underbrace{ \lambda^*\lt(\pa_z u_1^* \tau_{12}^\ee  + \pa_z v_1^\ee \sigma_{12}^*\rt) }_{\mathcal L_{22}'^\ee}+ \lambda^*\pa_z u_1^* \sigma_{12}^* -\sigma_{22}^*
\\
\tilde L_{12}^\ee =& \underbrace{-\dfrac{\lambda^*}2\lt(\ve \pa_x u_2^* (\tau_{11}^\ee-\tau_{22}^\ee) + \ve \pa_x v_2^\ee (\sigma_{11}^*-\sigma_{22}^*) +2\bv^\ee\cdot \nabla \sigma_{12}^*+2\bu^*\cdot \nabla \tau_{12}^\ee \rt)}_{\mathcal L_{12}^\ee}\\ 
&\qquad \qquad-\dfrac{\lambda^*}2\lt( 2\pa_t \sigma_{12}^*+2\bu^*\cdot \nabla \sigma_{12}^* + \pa_x u_2^*(\sigma_{11}^*-\sigma_{22}^*)\rt) + r\nu \ve \pa_x u_2^*,\\
\tilde L_{12}'^\ee =& \underbrace{-\dfrac{\lambda^*}2\lt(\pa_z u_1^* (\tau_{11}^\ee-\tau_{22}^\ee)  + \pa_z v_1^\ee (\sigma_{11}^*-\sigma_{22}^*)\rt) }_{\mathcal L_{12}'^\ee} +\dfrac{\lambda^*}2\pa_z u_1^* (\sigma_{11}^*-\sigma_{22}^*) -\sigma_{12}^* +r\nu \pa_z u_1^*;
\end{align*}
Note that the first order derivatives of $\bs^*$ occur in the terms $\tilde L^\ee$ and $C^\ee$. 
It will be shown in Theorem \ref{th:reg} that $\bs^*$ has sufficient regularity.

\nit Let us observe also that equations \eqref{EQ_rest_star} are similar to \eqref{EQ}, except for the linear terms on the right. Thus the energy estimates will be obtained similarly for both systems, multiplying Navier-Stokes equation by the velocity and Oldroyd equation by the stress tensor, and integrating over $\Om$.

\section{Limit equations}\label{ssec:limit}
\subsection{Limit system}
In an heuristic way, the following system of equations satisfied by $\bu^*$, $p^*$, $\bs^*$ is infered from \eqref{EQ_rest_star}:
$\bu^*$, $p^*$, $\bs^*$ are steady functions solutions of:
\begin{System}\label{EQ_star}
(1-r)\nu \pa_z^2 u_1^*-\pa_x p^* +\pa_z \sigma_{12}^*=0,\\
\pa_z p^*=0,\\
\nabla \cdot \bu^* =0,\\
\lambda^* \pa_z u_1^* \sigma_{12}^* + \sigma_{11}^*=0,\\
-\dfrac{\lambda^*}2 \pa_z u_1^*(\sigma_{11}^* -\sigma_{22}^*) + \sigma_{12}^* =r\nu \pa_z u_1^*,\\
-\lambda^* \pa_z u_1^* \sigma_{12}^* + \sigma_{22}^*=0.
\end{System}
This system is equipped with the following boundary condition (Dirichlet condition on the upper and lower part of the boundary, flux imposed on the lateral part of the boundary):
\begin{equation}\label{BC_star}
\begin{cases}
\bu^* =0\,, & \text{for $z=h(x)$},\\
\bu^* =(s,0)\,, & \text{for $z=0$},\\ 
\int\lm_0^{h(x)}\bu^*\,\di z\cdot n=\Phi_0 & \text{on $\Gamma_L$}.
\end{cases}
\end{equation}
The compatibility condition reads $\int\lm_{\pa \om} \Phi_0 =0$.
Moreover, since $p^*$ is defined up to a constant, the mean pressure is taken to be zero: $\intO p^*=0$.

\begin{rem}
 Each equation of the preceding system \eqref{EQ_star} is obtained by cancelling the constant part (i.e. the part independent of $\bv^\ee$, $q^\ee$, $\tau^\ee$) of respectively $C_1'$, $C_2'$, $\nabla \cdot \bu^*$, $\tilde L_{11}'^\ee$, $\tilde L_{12}'^\ee$, $\tilde L_{22}'^\ee$.
\end{rem}

\subsection{Determination of the boundary conditions}\label{ssec:BC}
\begin{rem}
 The lateral boundary conditions on $\bu^*$ do not depend on the ones on $\bu^\ee$, but only on the flux. Therefore, different boundary conditions on $\bu^\ee$ corresponding to the same flux lead to the same limit problem. This is a classical fact when passing from a two-dimensional problem to a one-dimensional problem (or similarly from a three-dimensional problem to a two-dimensional one), and has already been observed in \cite{BC86} for example. Here, in order to avoid boundary layers, $\bu^\ee=\bu^*$ is imposed on the lateral part of the boundary.

\nit Similarly, any value of $\bs^\ee$ on the boundary leads to the same limit problem. Again, in order to avoid boundary layers, well-prepared boundary conditions are also chosen for $\bs^\ee$.
\end{rem}

The preceding remark allows to define precisely the function $\bds J^\ve$ introduced in \eqref{BC}. Since $\bu^*|_{\Gamma_L}\in \H^{1/2}(\Gamma_L)$,
it is possible to construct $\bds J^\ve\in \H^{1/2}(\pa\Om)$ satisfying $\bds J^\ve|_{z=h}=0$, $\bds J^\ve|_{z=0}=(s,0)$ 
 and $\bds J^\ve|_{\Gamma_L} = \bu^*|_{\Gamma_L}$. 
Therefore, the boundary conditions on $\bu^\ee$ become
\begin{equation*}
 \begin{cases}
  \bu^\ee =0\,, & \text{for $z=h(x)$},\\
\bu^\ee =(s,0)\,, & \text{for $z=0$},\\  
\bu^\ee =\bu^* & \text{on $\Gamma_L$}.
 \end{cases}
\end{equation*}
Thus $\bu^\ee|_{\pa \Om} = \bu^*|_{\pa \Om}$, and $\bv^\ee$ will satisfy zero boundary conditions: $\bv^\ee|_{\pa \Om} =0$.

\nit Moreover, since $\bs^*\in \H^1(\Om)$ (see Theorem \ref{th:reg} for this regularity result), $\bds\theta^\ve$ can be defined as follows: 
\begin{equation}\label{deftheta}
 \bds\theta^\ve = \bs^*|_{\Gamma_+} \in \H^{1/2}(\Gamma_+).
\end{equation}
Therefore 
\begin{equation*}
\bs^\ee|_{\Gamma_+}=\bs^*|_{\Gamma_+},
\end{equation*}
and this implies that $\bt^\ee|_{\Gamma_+}=0$.

\nit On the other part $\Gamma_-$ of the boundary, $\bs^\ee$ is chosen such that $\bs^\ee\cdot n|_{\Gamma_-} = \bs^*\cdot n|_{\Gamma_-}$, for example $\bs^\ee|_{\Gamma_-}=\bs^*|_{\Gamma_-}$.

\subsection{Existence of a solution to the limit problem}\label{s:reg}
System \eqref{EQ_star}-\eqref{BC_star} has already been studied in \cite{BCM07}.
\begin{theo}\label{th:BCM}
Assume that $r<8/9$. Then system \eqref{EQ_star}-\eqref{BC_star} has a unique solution satisfying
\begin{equation}
 \bu^*\in \L^2(\Om),~~\pa_z \bu^* \in \L^2(\Om), ~~ p^* \in H^1(\om),~~\bs^* \in \L^2(\Om).
\end{equation}
\end{theo}
\begin{proof}
This result has been proved in  \cite{BCM07}.
\end{proof}

\nit This existence result is not sufficient for this study. Therefore, the following stronger regularity result is proved on the limit problem \eqref{EQ_star}-\eqref{BC_star}.
\begin{theo}\label{th:reg}
 Assume $r<2/9$. If $h\in H^k(\om)$, for $k\in \mathbb{N}^*$, then the unique solution $(\bu^*,p^*,\bs^*)$ of the system \eqref{EQ_star}-\eqref{BC_star} satisfies
\begin{equation}\label{regk}
\begin{gathered}
 p^* \in \C^{k+1}(\bar\om),\quad u_1^*,\,\pa_z u_1^*,\,\pa_z^2 u_1^*\in \C^{k+1}(\bar\Om), \quad \bs^*,\, \pa_z \bs^* \in \bds\C^{k+1}(\bar\Om),\\
\pa_x u_1^* \in \C^k(\bar\Om), \quad u_2^*,\,\pa_z u_2^*, \pa_z^2 u_2^* \in \C^k(\bar\Om), \quad\pa_x \bs^* \in  \bds\C^k(\bar\Om),\\
\pa_x u_2^* \in \C^{k-1}(\bar\Om).
\end{gathered}
\end{equation}
\end{theo}
\begin{proof}
\nit Let us observe that system \eqref{EQ_star} can be expressed as a system on $u_1^*$, $p^*$ only. Using~\eqref{EQ_star}, $\sigma_{11}^*$, $\sigma_{22}^*$ can be expressed as functions of $\sigma_{12}^*$ and $\pa_z u_1^*$. Indeed, from the fourth and the last equations of \eqref{EQ_star}, it holds that
\begin{equation}\label{alphagamma}
 \sigma_{22}^* = -\sigma_{11}^* = \lambda^* \pa_z u_1^* \sigma_{12}^*.
\end{equation}
Moreover, the divergence-free equation can be rewritten in order to eliminate $u_2^*$. Integrating this equation between $z=0$ and $z=h$, and using the fact that $u_2^*|_{z=0}=u_2^*|_{z=h}=u_1^*|_{z=h}=0$, it follows:
\begin{equation}\label{div_star}
\pa_x \lt( \int\lm_0^{h} u_1^* \,\di z \rt)=0.
\end{equation}
 Thus, the system in $u_1^*$, $p^*$ can be written in the following form:
\begin{equation}\label{EQ_reg}
\{ \begin{aligned}
& -\nu (1-r) \partial_z^2 u_1^* -\partial_z \sigma_{12}^* + \pa_x p^*=0, \qquad \text{with} \quad \sigma_{12}^*=\frac{\nu r \partial_z u_1^*}{1+\lambda^{*2}|\partial_z u_1^*|^2},\\
&\partial_z p^*=0,\\
& \pa_x \lt(\int\lm_0^{h} u_1^* \,\di z\rt)=0,
\end{aligned} \right.
\end{equation}
equipped with the boundary conditions stated in \eqref{BC_star} and the condition $\intO p^* =0$.

\nit For the sake of readability, the superscripts $^*$ are omitted in the rest of this section.

\nit Denote $q=\pa_x p$. Let $\phi \in \mathcal C^\infty(\R)$ defined by $\phi(t)=\nu (1-r) t + \dfrac{\nu r t}{1+\lambda^2t^2}$.
The first equation of \eqref{EQ_reg} becomes $q=\pa_z (\phi(\pa_z u_1))$.

\nit A simple study of function $\phi$ allows to show the following properties:
\begin{equation}\label{ptesphi}
0<\nu\lt(1-\frac{9r}8\rt) < |\phi'|_\infty<\nu , \qquad \text{and} \qquad \phi(t)\xrightarrow[t\rightarrow \pm \infty]{} \pm \infty.
\end{equation} 

\nit Therefore the function $\phi$ is invertible, and $\psi= \phi^{-1}$ belongs to $\mathcal C^\infty(\R)$. Moreover, $\psi$ is an increasing function as $\phi$. Integrating $ q=\pa_z (\phi(\pa_z u_1))$ with respect to $z$ between 0 and $z$, the first equation of \eqref{EQ_reg} becomes:
\begin{equation*}
 \phi(\pa_z u_1(x,z)) = q(x) \,z + \kappa(x),
\end{equation*}
where $\kappa(x)$ is a integration constant.
Therefore, it follows that
\begin{equation*}
 \pa_z u_1(x,z) = \psi(q(x)\, z +\kappa(x)).
\end{equation*}
Since $u_1|_{z=0}=s$, the integration between 0 and $z$ of the preceding equation yields:
\begin{equation}\label{expr_u}
 u_1(x,z)=s+ \int_0^z \psi(q(x)t + \kappa(x) ) \di t.
\end{equation}
The boundary condition $u_1|_{h(x)} =0$ implies also:
\begin{equation}\label{expr_int}
 \int_0^{h(x)} \psi(q(x)t+\kappa(x)) +s=0.
\end{equation}
For $(h,q,s,\kappa) \in \R^4$, let us introduce $F(h,q,s,\kappa)=\ds\int_0^{h} \psi(qt+\kappa) +s$.
\begin{lem}
For any $(h,q,s)\in \R^3$ there exists an unique $\kappa\in \R$ such that $F(h,q,s,\kappa)=0$.
\end{lem}
\begin{proof}
\begin{itemize}
 \item If such an $\kappa$ exists, it is unique from the implicit function theorem, since for all $(h,q,s,\kappa) \in~\R^4$
$$\dfrac{\pa F}{\pa \kappa}(h,q,s,\kappa) = \ds\int_0^h \psi'(qt+\kappa) \di t >0.$$ 
\item The following limits are computed, using the fact that $\lim\lm_{t\rightarrow \pm\infty} \psi(t)=\pm\infty$:
\begin{equation*}
 \lim_{\kappa\rightarrow +\infty} F(h,q,s,\kappa)=+\infty \quad \text{and} \quad \lim_{\kappa\rightarrow -\infty} F(h,q,s,\kappa)=-\infty.
\end{equation*}
Therefore, there exists $\kappa\in \R$ such that $F(h,q,s,\kappa)=0$. Let us denote $K(h,q,s)=\kappa$. By the implicit function theorem, $K\in \mathcal C^\infty(\R^3)$.
\end{itemize}
\end{proof}
Therefore, the following expression holds for $(h,q,s)\in \R^3$:
\begin{equation}\label{expr_F}
 F(h,q,s,K(h,q,s))=0.
\end{equation}
It is now possible to obtain an information on the sign of $\pa_q K$. Indeed, deriving the expression \eqref{expr_F} with respect to $q$, it follows
\begin{equation*}
\pa_q F + \pa_\kappa F \, \pa_q K = 0.
\end{equation*}
For $h>0$, since $\pa_q F = \ds\int_0^h t\psi'(qt+\kappa) \di t >0$ and $\pa_a F = \ds\int_0^h \psi'(qt+\kappa) \di t >0$, $\pa_q K$ is strictly negative.

\nit Now, using equation \eqref{div_star} and the expression \eqref{expr_u} for $u$, it follows:
\begin{align*}
 &\int_0^{h(x)} \int_0^z \pa_x \Big(\psi(q(x)t+K(h(x),q(x),s)) \Big) \di t \,\di z=0.\\
\intertext{or if changing the direction of integration}
 &\int_0^{h(x)} (h(x)-t) \pa_x \Big(\psi(q(x)t+K(h(x),q(x),s)) \Big) \di t=0.
\end{align*}
This can be rewritten as
\begin{equation}
\begin{split}
 q'(x) &\int_0^{h(x)} (h(x)-t) \, \Big((t+\pa_q K(h(x),q(x),s)\Big)\, \psi'\big(q(x)t + K(h(x),q(x),s)\big) \di t \\
&=- \int_0^{h(x)} (h(x)-t) \, \Big(h'(x)\pa_h K(h(x),q(x),s)\Big) \,\psi'\big(q(x)t + K(h(x),q(x),s)\big) \di t,
\end{split}\nonumber
\end{equation}
which can be seen as an ordinary differential equation in $q$. Let 
\begin{equation}
\begin{split}
U (x,q)&=\int_0^{h(x)} \Big(h(x)-t\Big) \, \Big(t+\pa_q K(h(x),q,s)\Big)\, \psi'\big(qt + K(h(x),q,s)\big) \di t,\\
V(x,q)&=\int_0^{h(x)} \Big(h(x)-t\Big) \, \Big(h'(x)\pa_h K(h(x),q,s)\Big) \,\psi'\big(qt + K(h(x),q,s)\big) \di t.
\end{split}\nonumber 
\end{equation}
The differential equation becomes $U(x,q(x)) \, q'(x) = -V(x,q(x))$ for $x\in \om$.
Note that this equation is in some sense a generalized Reynolds equation for the pressure.
\begin{lem}\label{signU}
 Let $r<2/9$. Then $U(x,q)<0$ for any $(x,q)\in \om \times\R$.
\end{lem}
\begin{proof}
Let $(x,q) \in \om \times \R$. Equation \eqref{expr_int} and the definition \eqref{expr_F} of $K$ imply:
\begin{equation*}
\int_0^{h(x)} \psi(qt + K(h(x),q,s) ) \di t = -s,
\end{equation*}
which becomes, after derivation with respect to $q$
\begin{equation}\label{expr_interm}
 \int_0^{h(x)} \Big(t+\pa_q K(h(x),q,s)\Big) \psi'\big(qt + K(h(x),q,s) \big) \di t = 0.
\end{equation}
With the notation $K'(x,q)=\pa_q K(h(x),q,s)$, \eqref{expr_interm} implies
\begin{equation*}
 K'(x,q)=-\frac{\ds\int_0^{h(x)} t \, \psi'\big(qt + K(h(x),q,s)\big) \di t}{\ds\int_0^{h(x)} \psi'\big(qt + K(h(x),q,s)\big) \di t}.
\end{equation*}
Now, using this expression, $U(x,q)$ can be simplified:
\begin{equation}\label{expr_U}
 U (x,q)=\int_0^{h(x)} -t \, \Big(t+\pa_q K(h(x),q,s)\Big)\, \psi'\big(qt + K(h(x),q,s)\big) \di t.
\end{equation}
Recalling the estimate of $|\phi|_\infty$ in \eqref{ptesphi}, it follows that for any $t\in \R$:
\begin{equation*}
 \dfrac1\nu < \psi'(t)= \dfrac1{\phi'(\psi(t)} < \dfrac1{\nu (1-9r/8)}
\end{equation*}
Let $m= \dfrac1\nu$, $M=\dfrac1{\nu (1-9r/8)}$.
Then 
\begin{equation*}
 -\dfrac{bh(x)}{2m}\leq K'(x,q)\leq -\dfrac{ah(x)}{2M}.
\end{equation*}
Now, \eqref{expr_U} implies that:
\begin{align*}
h(x)^3\lt(\frac{m}3 -\frac{M}4\rt)&= \int_0^{h(x)} tm\lt(t-\frac{Mh(x)}{2m}\rt) \\
&\leq -U(x,q) \leq \int_0^{h(x)} tM\lt(t-\frac{mh(x)}{2M}\rt) =h(x)^3\lt(\frac{M}3 -\frac{m}4\rt).
\end{align*}
In order to prove that $U$ remains strictly negative, it suffices to prove that $0<\dfrac{m}3 -\dfrac{M}4$, i.e. that $\dfrac{m}M > \dfrac34$, which is satisfied under the condition $r<\dfrac29$.
\end{proof}

\nit It is possible to apply Picard-Lindel\"of theorem (or Cauchy-Lipschitz theorem) to the ordinary differential equation $-U(x,q(x))\,q'(x)=V(x,q(x))$, as $U$ remains strictly negative by Lemma \ref{signU}. Since $\psi $ and $K$ are $\mathcal C^\infty$-functions, the regularity of $q'$ is determined by the regularity of $q$ and $h$. By hypothesis, $h$ belongs to $H^k(\om)$, with $k\in \mathbb{N}$, hence $h\in L^2(\om)$. Moreover, Theorem \ref{th:BCM} implies that $q\in L^2(\om)$.  Thus $q'\in L^2(\om)$, which means $q\in H^1(\om)$.

\nit Iterating this process as long as $h$ is regular, $h\in H^k(\om)$ and $q \in H^k(\om)$ implies that $q'\in H^k(\om)$, thus $\pa_x p=q\in H^{k+1}(\om)$, and $p\in H^{k+2}(\om)$. By the classical Sobolev embedding, $p$ belongs to $\mathcal C^{k+1}(\bar\om)$. 

\nit Last, recalling the expression \eqref{expr_u}, it follows that $u_1\in \mathcal C^{k+1}(\bar \om)$, and, taking the first and second derivatives of \eqref{expr_u} with respect to $z$, that $\pa_z u_1,~ \pa_z^2 u_1$ also belong to $\mathcal C^{k+1}(\bar \om)$.

\nit As observed in the introduction of the proof, $\bs$ and $u_2$ are given as functions of $p$, $u_1$, and the needed regularity follows.
\end{proof}

\begin{rem}
 Since in practical applications, $h$ is very regular ($h\in C^\infty(\bar\om)$), the preceding theorem gives as much regularity as wanted. In particular, the following result will be useful subsequently.
\end{rem}

\begin{cor}
 Assume $r<2/9$. If $h\in H^1(\om)$, then the unique solution $(\bu^*,p^*,\bs^*)$ of the system \eqref{EQ_star}-\eqref{BC_star} satisfies
\begin{equation}\label{reg}
\begin{gathered}
 p^* \in \C^2(\bar\om),\quad u_1^*,\,\pa_z u_1^*,\,\pa_z^2 u_1^*\in \C^2(\bar\Om), \quad \bs^*,\, \pa_z \bs^* \in \bds\C^2(\bar\Om),\\
\pa_x u_1^* \in \C^1(\bar\Om), \quad u_2^*,\,\pa_z u_2^*, \pa_z^2 u_2^* \in \C^1(\bar\Om), \quad\pa_x \bs^* \in  \bds\C^1(\bar\Om),\\
\pa_x u_2^* \in \C^0(\bar\Om).
\end{gathered}
\end{equation}
\end{cor}
\begin{proof}
 It suffices to take $k=1$ in the preceding theorem \ref{th:reg}.
\end{proof}

\section{Convergence of the remainders}\label{remainders}

\subsection{Equations on the remainders}
From now on, the superscript $^\ee$ are dropped although the functions still depend on $\ve$ and $\eta$. Using the equations \eqref{EQ_star}, system \eqref{EQ_rest_star} becomes
\begin{EqSystem*}[EQ_rest]
 \rho ~\di_t v_1 - (1-r) \nu\,\Delta_\ve v_1 + \dfrac1{\ve^2}\pa_x q -\dfrac1\ve \pa_x \tau_{11} - \dfrac1{\ve^2} \pa_z \tau_{12}= L_1+\dfrac1\ve C_1,\\
\rho ~\di_t v_2 - (1-r) \nu\,\Delta_\ve v_2 + \dfrac1{\ve^4}\pa_x q -\dfrac1{\ve^2} \pa_x \tau_{12} - \dfrac1{\ve^3} \pa_z \tau_{22} = \dfrac1{\ve^2} L_2+\dfrac1{\ve^3} C_2,\\
\nabla \cdot \bv = 0,\\
\lambda^* \di_t \tau_{11} -\lambda^* N(\bv,\tau_{12}) + \dfrac1\ve \tau_{11} -\eta \Delta_\ve \tau_{11}-2r\nu \pa_x v_1=  L_{11} + \dfrac1\ve  L_{11}'+\eta\Delta_\ve \sigma_{11}^* ,\\
\lambda^* \di_t \tau_{12} +\dfrac{\lambda^*}2 N(\bv,\tau_{11}-\tau_{22})\ + \dfrac1\ve \tau_{12} -\eta \Delta_\ve \tau_{12}-r \nu \lt(\pa_x v_2+\dfrac1\ve \pa_z v_1\rt) =  L_{12} + \dfrac1\ve  L_{12}'+\eta\Delta_\ve \sigma_{12}^*,\\
\lambda^* \di_t \tau_{22} +\lambda^* N(\bv,\tau_{12}) + \dfrac1\ve \tau_{22} -\eta \Delta_\ve \tau_{22}-\dfrac{2r\nu}\ve \pa_z v_2 =  L_{22} + \dfrac1\ve L_{22}'+\eta\Delta_\ve \sigma_{22}^*,
\end{EqSystem*}
with the new quantities
\allowdisplaybreaks\begin{align*}
L_1&= \mathcal L_1-\rho ~\bu^*\cdot \nabla u_1^* + (1-r)\nu \pa_x^2 u_1^*,\\
L_2&=\mathcal L_2-\rho ~\ve^2\bu^*\cdot \nabla u_2^* + (1-r)\nu \pa_x^2 u_2^*+(1-r)\nu \pa_z u_2^* + \pa_x \sigma_{12}^*,\\
L_{11}& = \mathcal L_{11} + \lambda^*\lt(- \bu^*\cdot \nabla \sigma_{11}^* +\ve \pa_x u_2^* \sigma_{12}^*\rt) + 2r\nu \pa_x u_1^*,\\
L_{11}'& =\mathcal L_{11}',\\
L_{12} &= \mathcal L_{12}-\dfrac{\lambda^*}2\lt( 2\bu^*\cdot \nabla \sigma_{12}^*+\pa_x u_2^*(\sigma_{11}^*-\sigma_{22}^*)\rt) + r\nu \ve \pa_x u_2^*,\\
L_{12}' &= \mathcal L_{12}',\\
L_{22} &= \mathcal L_{22} -\lambda^*\lt( \bu^*\cdot \nabla \sigma_{22}^* +\ve \pa_x u_2^* \sigma_{12}^*\rt) + 2r\nu \pa_z u_2^*,\\
L_{22}' &= \mathcal L_{22}'.
\end{align*}
and with the initial and boundary conditions 
\begin{equation}\label{BCIC}
\bv|_{t=0} = \bu_0-\bu^*,\quad  ~\bt|_{t=0} = \bs_0-\bs^*,\quad ~\bv|_{\pa \Om} =0, \quad ~\bt|_{\Gamma_+} =0.
\end{equation}
Let us observe that both initial conditions $\bv|_{t=0}$ and $\bt|_{t=0}$ belong to $\L^2(\Om)$.
$\bv$, $q$ and $\bt$ are defined by \eqref{dvasymptU}, \eqref{dvasymptp}, \eqref{dvasympts}. From the existence theorem \ref{th:LM} for $(\bu,p,\bs)$ and theorem \ref{th:BCM} for $(\bu^*,p^*,\bs^*)$, it follows that system \eqref{EQ_rest} admits a solution $(\bv,q,\bt)\in L^2(0,\infty,\H^1(\Om)) \times L^2(0,\infty,L^2(\Om)) \times \mathcal C (0,\infty,\L^2(\Om)) $  for $r<8/9$.

\subsection{Convergence of $\bv$ and $\bt$}
Before starting the a priori estimates, let us explain how the non-linear terms in \eqref{EQ_rest} are handled. The non-linear terms $\bv\cdot \nabla\bv$ of Navier-Stokes equation and $\bv\cdot \nabla \bt$ of Oldroyd equation are treated with the following Lemma \ref{lm:convct}. On the other hand, the non-linear terms $N(\bv,\bt)=\lt(\ve \pa_x v_2 -\frac1\ve \pa_z v_1\rt)\bt	$ in \eqref{EQ_rest:d}-\eqref{EQ_rest:f} are zero when multiplied by $\bt$.
\begin{lem}\label{lm:convct}
Let $\bn$ be the exterior normal of the domain $\Om$. Let $\bds\phi\in \H^1(\Om)$ be a vector field satisfying $\nabla \cdot \bds\phi=0$ and $\bds\phi\cdot \bn|_{\pa \Om}=0$. Let $w\in H^1(\Om)$.
Then 
\begin{equation*}
 \intO \bds\phi\cdot \nabla w \,w =0.
\end{equation*}
\end{lem}
\begin{proof}
By integration by parts:
\begin{equation*}
 \intO \bds\phi\cdot \nabla w \,w = -\intO \underbrace{\nabla \cdot \bds\phi}_{=0}\cdot w^2 -\intO \bds\phi\cdot \nabla w \,w + \int\lm_{\pa\Om} \underbrace{\bds\phi\cdot \bn}_{=0}\, w^2 =0.
\end{equation*}
\end{proof}

\nit The classical approach consists in obtaining a priori estimates for $\bv$. 
\begin{prop}
Let $(\bv,q,\bt)$ be a solution of \eqref{EQ_rest}.
Then $\bv=(v_1$, $v_2)$ satisfy the following inequality for $\ve$ small enough:
\begin{equation}\label{apriori-r1r2}
  r\nu \rho \dfrac{\di}{\di t} \lt(|v_1|^2 +|\ve v_2|^2\rt) + \frac32 r(1-r)\nu^2 \lt(|\nabla_\ve v_1|^2 + |\ve\nabla_\ve v_2|^2\rt) 
\leq-\mathcal D_1 -\mathcal D_2 +C, 
\end{equation}
where $\mathcal D_1 = \dfrac{2r\nu}\ve \ds\intO \a\,\pa_x v_1 +\dfrac{2r\nu}{\ve^2} \intO \b\,\pa_z v_1$, $\mathcal D_2 =2r\nu\ds \intO \b\,\pa_x v_2 +\dfrac{2r\nu}{\ve} \intO \c\,\pa_z v_2$ and $C$ is a constant independent of $\ve$.
\end{prop}
\begin{proof}
The proof consists in obtaining classical a priori estimates on both $v_1$ and $v_2$.
\begin{asparaenum}[\bf Step 1.]
 \item Let us multiply \eqref{EQ_rest:a} by $ v_1$ and integrate over $\Om$. Observe that $v_1$ is regular enough to do so. Since $\bv|_{\pa \Om} = 0$, the boundary terms in the integration by parts are all zero. For example $-\intO \Delta_\ve v_1\, v_1 = \intO |\nabla_\ve v_1| ^2$. Moreover, the convection terms $\intO \bv\cdot \nabla v_1 \,v_1$ contained in $\intO \di_t v_1\,v_1$ are equal to zero by Lemma \ref{lm:convct}, since $\nabla \cdot \bv=0$ and $\bv|_{\pa \Om}=0$. It follows:
\begin{equation}\label{apriori-r1}
 \frac\rho2 \dfrac{\di}{\di t} |v_1|^2+ (1-r)\nu|\nabla_\ve v_1|^2 -\dfrac1{\ve^2} \intO q\, \pa_x v_1= \underbrace{-\dfrac1\ve \intO \a\,\pa_x v_1 -\dfrac1{\ve^2} \intO \b\,\pa_z v_1}_{-\mathcal D_1/2r\nu} +\intO L_1\,v_1 + \dfrac1\ve\intO C_1\,v_1.
\end{equation}

It remains to estimate the terms $\intO L_1\,v_1$ and $\intO C_1\,v_1$. 

\paragraph{Main idea} 
{Estimates of the form: $ \intO L_1\,v_1 +\frac1\ve\intO C_1\,v_1 \leq C+ \kappa_1 |\nabla_\ve v_1|^2 + \kappa_2|\pa_z v_2|^2$ will be proved, where $C$ is a constant independent of $\ve$ and where the constants $\kappa_1, \kappa_2$ satisfy $\kappa_1,\kappa_2 < (1-r)\nu/4$. These constants will be precised later in the proof.} 

\nit In the following, $C$, $c_i$ and $M_i$ will denote some constants independent of $\ve$ and $\eta$, which might depend on $|\Om|$, on the physical parameters of the problem and on $\bu^*$, $\bs^*$ in sufficiently regular norms. 
\begin{itemize}
 \item Let us estimate first the linear (with respect to $\bv$) term $\mathcal L_1$ of $L_1$. To this end, Poincaré inequality is useful: for $f\in L^2(\Om)$, with $f|_{z=h}=0$, $|f|\leq C_P |\pa_z f|$. The constant $C_P$ only depends on $\Om$. 
\begin{asparaitem}[$\star$]
 \item $\rho\ds\intO v_1\,\pa_x u_1^*\,v_1 \leq \rho|\pa_x u_1^*|_\infty\, |v_1|^2 \leq \rho \,\ve^2 C_P^2|\pa_x u_1^*|_\infty\, \lt|\dfrac1\ve\pa_z v_1\rt|^2=:M_1 \ve^2 \lt|\dfrac1\ve\pa_z v_1\rt|^2$.

\nit Note that by Theorem \ref{th:reg}, $\pa_x u_1^* \in L^\infty(\Om)$. In the following, all the regularity results used in the estimates also follow from Theorem \ref{th:reg}.

\item For the next term, Poincaré inequality is combined with Young inequality:
\begin{equation}
\begin{split}
 \rho\ds\intO v_2\, \pa_z u_1^*\,v_1 \leq \rho|\pa_z u_1^*|_\infty\,|v_2|\, |v_1|& \leq \rho \,C_P^2|\pa_z u_1^*|_\infty\,|\pa_z v_2|\, |\pa_z v_1|\\
\vspace{-1cm}
& \leq \underbrace{\rho \,C_P^2|\pa_z u_1^*|_\infty}_{=:M_2} \lt(\dfrac\ve2 |\pa_z v_2|^2 + \dfrac\ve2 \lt|\dfrac1\ve\pa_z v_1\rt|^2\rt).
\end{split}\nonumber
\end{equation}

\item In a similar way:
\begin{equation*} 
\rho\intO \bu^*\cdot\nabla v_1 \,v_1 \leq \underbrace{\rho\, C_P|u_1^*|_\infty}_{=:M_3}\lt(\dfrac\ve2 |\pa_x v_1|^2 + \dfrac\ve2 \lt|\dfrac1\ve\pa_z v_1\rt|^2\rt)+\ve^2 \underbrace{\rho\, C_P|u_2^*|_\infty }_{=:M_4}\,\lt|\dfrac1\ve \pa_z v_1\rt|^2.
\end{equation*}
Observe here that it was not possible to apply Lemma \ref{lm:convct}, since $\bu^*\cdot \bn|_{\pa\Om} \neq 0$.
\end{asparaitem}

\item It remains the easier terms of $L_1$ and $C_1$ (the ones which do not depend on $\bv$).
\begin{asparaitem}[$\star$]
\item The first term is treated using again Poincaré and Young inequalities:
\begin{equation*} 
\rho\intO \bu^*\cdot \nabla u_1^* \,v_1 \leq \rho \,C_P |\bu^*|_\infty\,|\nabla u_1^*|\, |\pa_z v_1| \leq \frac12(\rho\, C_P |\bu^*|_\infty\,|\nabla u_1^*|)^2+ \dfrac{\ve^2}2  \lt|\dfrac1\ve\pa_z v_1\rt|^2 \leq C +\dfrac{\ve^2}2  \lt|\dfrac1\ve\pa_z v_1\rt|^2.
\end{equation*}

\item Similarly, $(1-r)\nu\ds \intO \pa_x^2 u_1^* \, v_1 \leq C +  \dfrac{\ve^2}2  \lt|\dfrac1\ve\pa_z v_1\rt|^2$.

\item The last term is estimated as follows, using Young inequality:
\begin{equation*}
 \dfrac1\ve \intO \pa_x \sigma_{11}^* \, v_1 \leq \frac1{4c}|\pa_x \sigma_{11}^*|^2 + c\lt|\dfrac1\ve\pa_z v_1\rt|^2 \leq C+c\lt|\dfrac1\ve\pa_z v_1\rt|^2,
\end{equation*}
where $c$ is a positive constant independent of $\ve$ that can be chosen arbitrarily.
\end{asparaitem}
Now, let us choose $\ve$ and $c$ small enough such that all constants satisfy: 
\begin{equation} \label{epetit1}
M_1\ve^2, \frac{M_2 \ve}2, \frac{M_3\ve}2, M_4\ve^2, \frac{\ve^2}2, c \leq \frac{(1-r)\nu}{36}.
\end{equation}
\end{itemize}

\item Let us multiply \eqref{EQ_rest:b} by $\ve^2 v_2$ and integrate over $\Om$. Again, the boundary terms in the integrations by parts vanish, since $v_2|_{\pa\Om}=0$, and the convection terms are equal to zero since $\nabla \cdot \bv=0$ and $\bv|_{\pa \Om} =0$ (by Lemma \ref{lm:convct}). It follows:
\begin{equation}\label{apriori-r2}
\frac{\rho\,\ve^2}2 \dfrac{\di}{\di t} |v_2|^2 + (1-r)\nu |\ve\nabla_\ve v_2|^2 -\dfrac1{\ve^2} \intO q\, \pa_z v_2
 = \underbrace{- \intO \b\,\pa_x v_2 -\dfrac1{\ve} \intO \c\,\pa_z v_2}_{-\mathcal D_2/2r\nu} + \intO L_2\,v_2 + \dfrac1\ve\intO C_2\,v_2.
\end{equation}
Each term of $\intO L_2\,v_2$ and $\intO C_2\,v_2$ is estimated with the help of Poincaré and Young inequalities as in the preceding step. 
\begin{asparaitem}[$\star$]
 \item $\ve^2\rho\ds \intO \bv\cdot \nabla u_2^*\, v_2 \leq \ve^2\underbrace{\rho C_P^2|\pa_x u_2^*|_\infty}_{=:M_5}\lt(\dfrac\ve2 \lt|\dfrac1\ve \pa_z v_1\rt|^2+ \dfrac\ve2 |\pa_z v_2|^2 \rt) +\ve^2 \underbrace{\rho C_P^2 |\pa_z u_2^*|_\infty}_{=:M_6}\,|\pa_z v_2|^2 $.

\item $\ve^2\rho\ds\intO \bu^*\cdot \nabla v_2 \,v_2 \leq \ve\underbrace{\rho C_P|u_1^*|_\infty}_{=:M_7} \lt(|\ve \pa_x v_2|^2 + |\pa_z v_2|^2\rt) + \ve^2 \underbrace{\rho C_P|u_2^*|_\infty}_{=:M_8} |\pa_z v_2|^2$.

\item $\ve^2\rho\ds\intO \bu^*\cdot\nabla u_2^* \,v_2  \leq \frac12\ve^2 \rho |\bu^*|_\infty^2 |\nabla u_2^*|^2+ \ve^2 \underbrace{\frac12 C_P^2}_{=:M_9}|\pa_z v_2|^2 \leq C+ \ve^2 M_9 |\pa_z v_2|^2$.

\item By integration by parts (all boundary terms are equal to zero since $v_2|_{\pa\Om} =0$) and Young inequality as before:
 $$(1-r)\nu\ve^2\ds\intO \pa_x^2 u_2^* \,v_2 = -(1-r)\nu\ve^2\intO \pa_x u_2^* \,\pa_x v_2 \leq \ve\underbrace{(1-r)\nu}_{=:M_{10}}\lt(\frac12|\pa_x u_2^*|^2 + \frac12|\ve \pa_x v_2|^2\rt).$$

\item $(1-r)\nu\ds\intO \pa_z^2 u_2^* \, r_2 \leq \frac1{4c_1}(1-r)^2\nu^2 C_P^2|\pa_z^2 u_2^*|^2 + c_1|\pa_z v_2|^2 \leq C+ c_1|\pa_z v_2|^2$, where $c_1$ is a arbitrary positive constant.
\item $\ds\intO \pa_x \sigma_{12}^* \, v_2 \leq \frac{C_P^2}{4c_1} |\pa_x\sigma_{12}^*|^2 + c_1|\pa_z v_2|^2 \leq C+ c_1|\pa_z v_2|^2 $.

\item The $C_2$ term is treated with integration by parts (again, no boundary terms since $v_2|_{\pa\Om}=v_1|_{\pa\Om} =0$) and the divergence equation. The term is then treated as the preceding one:
\begin{equation} 
\begin{split}
\dfrac1\ve \intO \pa_z \sigma_{22}^* \, v_2 = -\dfrac1\ve \intO \sigma_{22}^* \, &\pa_z v_2 = \dfrac1\ve \intO\sigma_{22}^*\, \pa_x v_1 = -\dfrac1\ve \intO \pa_x \sigma_{22}^* v_1 \\
&\leq C_P|\pa_x \sigma_{22}^*|\,\lt|\dfrac1\ve \pa_z v_1\rt| \leq \frac{C_P^2}{4c_2}|\pa_x \sigma_{22}^*|^2+c_2\lt|\dfrac1\ve \pa_z v_1\rt|^2 \leq C+c_2\lt|\dfrac1\ve \pa_z v_1\rt|^2 .
\end{split}\nonumber
\end{equation}
\end{asparaitem}
Now, let us choose $\ve$, $c_1$ and $c_2$ small enough such that 
\begin{equation}\label{epetit2}
\frac{M_5 \ve^3}2, M_6\ve^2, M_7\ve, M_8\ve,M_9\ve, \frac{M_{10}\ve}2, c_1,c_2 \leq \frac{(1-r)\nu}{36}.
\end{equation}

\item After summing \eqref{apriori-r1} and \eqref{apriori-r2}, and multiplying by $2r\nu$, it holds for $\ve$ small enough (satisfying \eqref{epetit1} and \eqref{epetit2}):
\begin{equation*}
  r\nu \rho \dfrac{\di}{\di t} \lt(|v_1|^2 +|\ve v_2|^2\rt) + \frac32 r(1-r)\nu^2 \lt(|\nabla_\ve v_1|^2 + |\ve\nabla_\ve v_2|^2\rt) -\dfrac{2r\nu}{\ve^2} \intO q\,\lt( \pa_x v_1 + \pa_z v_2\rt)\leq -\mathcal D_1 -\mathcal D_2 +C,
\end{equation*}
where $C$ is a constant independent of $\ve$.
From the divergence equation $\nabla \cdot \bv =  \pa_x v_1 + \pa_z v_2=0$ it follows that the pressure term $\intO q\,\lt( \pa_x v_1 + \pa_z v_2\rt)=0$, and equation \eqref{apriori-r1r2} is obtained.
\end{asparaenum}
\end{proof}

\begin{prop}
 Let us suppose that
\begin{equation*}
\lambda^* |\pa_z u_1^*|_\infty\leq 1/12,~ \lambda^* |\sigma_{12}^*|_\infty\leq \chi, ~\lambda^*(|\sigma_{11}^*|_\infty +|\sigma_{22}^*|_\infty)\leq \chi, ~ 2\lambda^*|\pa_z \sigma_{12}^*|_\infty \leq \chi, ~\lambda^* |\pa_z \sigma_{11}^*|_\infty \leq \chi, \end{equation*}
where $\chi =\dfrac\nu6\sqrt{r(1-r)}$.
Then for $\ve$ small enough, $\a$, $\b$, $\c$ solution of \eqref{EQ_rest} satisfy the following inequality:
\begin{equation}\label{apriori-abc}
\begin{split}
  \dfrac{\lambda^*}{2\ve}&\dfrac{\di}{\di t} \lt(|\tau_{11}|^2 + 2|\tau_{12}|^2 +|\tau_{22}|^2\rt) + \frac12\lt(\lt|\dfrac1\ve \a\rt|^2+ 2\lt|\dfrac1\ve \b\rt|^2+ \lt|\dfrac1\ve \c\rt|^2\rt)\\
&+ \frac\eta\ve \lt(|\nabla_\ve \tau_{11}|^2 + 2 |\nabla_\ve \tau_{12}|^2 + |\nabla_\ve \tau_{22}|^2\rt)\leq \mathcal D_1+ \mathcal D_2 + r(1-r)\nu^2\lt(|\nabla_\ve v_1|^2+|\ve \nabla_\ve v_2|^2\rt)+C,
\end{split}
\end{equation}
where $C$ is a constant independent of $\ve$.
\end{prop}
\begin{proof}
As in the preceding proposition, classical a priori estimates on $\a$, $\b$ and $\c$ are obtained, and the remaining terms are estimated accurately.
\begin{asparaenum}[\bf Step 1.]
\item Let us multiply \eqref{EQ_rest:d} by $\dfrac{\tau_{11}}\ve$ and integrate over $\Om$.
Again, the convection terms $\intO \bv\cdot \nabla\a\,\a$ contained in $\intO\di_t \a\,\a$ are equal to zero by Lemma \ref{lm:convct}, since $\nabla \cdot \bv=0$ and $\bv|_{\pa \Om}=0$ (see \eqref{BCIC}). Moreover, there is no boundary term in the integration by parts since the boundary conditions on $\bs$ have be chosen such that $\bt\cdot n|_{\pa\Om}=0$ (see also \eqref{BCIC}). It follows:
\begin{equation}\label{apriori-a}
\begin{split}
 \dfrac{\lambda^*}{2\ve} \dfrac{\di}{\di t} |\tau_{11}|^2 -\dfrac{\lambda^*}\ve \intO N(\bv,\tau_{12})\,\tau_{11}& + \lt|\dfrac1\ve \a\rt|^2
+\frac\eta\ve |\nabla_\ve \tau_{11}|^2\\
&=\dfrac{2r\nu}\ve \intO \pa_x v_1\,\tau_{11}+\dfrac1\ve \intO L_{11}\, \tau_{11} + \dfrac1{\ve^2} \intO L_{11}' \, \tau_{11}.
\end{split}
\end{equation}
\begin{itemize}
\item The terms of $\ds \intO \mathcal L_{11}\, \a$ are estimated as follows:
\begin{asparaitem}[$\star$]
\item $\lambda^* \ds\intO \pa_x u_2^* \,\b\, \tau_{11} \leq \underbrace{\lambda^*|\pa_x u_2^*|_\infty}_{=:M_{11}}\lt(\dfrac{\ve^2}2\lt|\dfrac1\ve \a\rt|^2 + \dfrac{\ve^2}2\lt|\dfrac1\ve \b\rt|^2\rt)$. 

\item In a same way:
$\dfrac{\lambda^*}\ve \ds\intO v_1\,\pa_x \sigma_{11}^*\, \tau_{11} \leq \underbrace{\lambda^*|\pa_x \sigma_{11}^*|_\infty C_P}_{=:M_{12}} \lt(\dfrac{\ve}2\lt|\dfrac1\ve \pa_z v_1\rt|^2 + \dfrac\ve2 \lt|\dfrac1\ve \a\rt|^2\rt)$.

Let us choose $\ve$ small enough such that:
\begin{equation*}
 \frac{M_{11}\ve^2}2 \leq \frac1{24} \qquad \text{and} \qquad  \frac{M_{12}\ve}2 \leq \text{Min}\{\frac{r(1-r)\nu}6, \frac1{24}\}.
\end{equation*}

\item $\lambda^* \ds \intO \pa_x v_2\,\sigma_{12}^* \, \tau_{11} \leq \lambda^* |\sigma_{12}^*|_\infty\, |\ve\pa_x v_2|\,\lt|\dfrac1\ve \a\rt| \leq \lambda^* |\sigma_{12}^*|_\infty \lt(\frac1{4c_3}  |\ve\pa_x v_2|^2 + c_3  \lt|\dfrac1\ve \a\rt|^2\rt)$. 

Here, it is not possible to choose $c_3$ such that both coefficients are less than $r(1-r)\nu/6$ and $1/24$. Therefore,a condition on $\lambda^* |\sigma_{12}^*|_\infty$ is imposed such that:
\begin{equation*}
 \frac{\lambda^* |\sigma_{12}^*|_\infty }{4c_3} \leq \frac{r(1-r)\nu}6 \qquad \text{and} \qquad \lambda^* |\sigma_{12}^*|_\infty c_3 \leq \frac1{24}.
\end{equation*}
Choosing $c_3$ satisfying $\lambda^*|\sigma_{12}^*|_\infty c_3 = 1/24$, the condition on $\lambda^* |\sigma_{12}^*|_\infty$ becomes:
\begin{equation*}
 \lambda^* |\sigma_{12}^*|_\infty \leq \frac\nu6 \sqrt{r(1-r)}=:\chi.
\end{equation*}

\item Similarly the following term can be estimated:
\begin{equation*}
 \dfrac{\lambda^*}\ve \ds \intO v_2\,\pa_z \sigma_{11}^* \, \tau_{11} \leq \lambda^*|\pa_z \sigma_{11}^*|_\infty \,|\pa_z v_2|\,\lt|\dfrac1\ve \a\rt| \leq \lambda^*|\pa_z \sigma_{11}^*|_\infty\lt(\frac1{4c_3} |\pa_z v_2|^2 + c_3 \lt|\dfrac1\ve \a\rt|^2\rt).
\end{equation*}
The same reasoning as before allows to control both terms providing that $ \lambda^*|\pa_z \sigma_{11}^*|_\infty \leq \chi$.

\item In order to treat the term $-\lambda^*\ds \intO \bu^*\cdot\nabla \tau_{11} \,\tau_{11} $, 
it is not possible to apply Lemma \ref{lm:convct}, since $\bu^*\cdot n|_{\pa \Om} \neq 0$. However, integration by parts implies that
\begin{equation*}
-\lambda^*\intO \bu^*\cdot\nabla \tau_{11} \,\tau_{11} = -\frac{\lambda^*}2\int\lm_{\pa \Om} \bu^*\cdot n \, \tau_{11}^2.
\end{equation*}
On $\om$, since $\bu^*=(s,0)$ (see \eqref{BC_star}), it holds $\bu^*\cdot n =0$. Thus it remains to consider the boundary integral on $\Gamma_L$. This boundary integral is split into two integrals on $\Gamma_+$ and $\Gamma_-$. On $\Gamma_-$, it holds $ \bu^*\cdot n >0$, thus $-\frac{\lambda^*}2\int\lm_{\Gamma_-} \bu^*\cdot n \, \tau_{11}^2 \leq 0$, and this term is trivially bounded by zero.
On $\Gamma_+$, the boundary conditions are chosen in subsection \ref{ssec:BC} such that $\bt|_{\Gamma_+} =0$, therefore $-\frac{\lambda^*}2\int\lm_{\Gamma_+} \bu^*\cdot n \, \tau_{11}^2 =0$.
\end{asparaitem}

\item All other terms of $\ds \intO L_{11}\, \a$ are easier to manage, since they are linear in $\a$, and they are treated with Young and Poincaré inequalities in a same way as the ones in $v_1$, $v_2$.

\item For the terms of $\ds \intO L_{11}'\, \a$, we proceed as before:
\begin{equation*}
 \dfrac{\lambda^*}{\ve^2}\ds \intO \pa_z u_1^*\, \b\, \tau_{11} \leq \lambda^* |\pa_z u_1^*|_\infty\, \lt|\dfrac1\ve \b\rt|\, \lt|\dfrac1\ve \a\rt| \leq \lambda^* |\pa_z u_1^*|_\infty\lt(\frac12 \lt|\dfrac1\ve \b\rt|^2 + \frac12\lt|\dfrac1\ve \a\rt|^2\rt).
\end{equation*}
Choosing $\lambda^* |\pa_z u_1^*|_\infty \leq 1/12$, both terms are bounded by $1/24$.

\begin{equation*}
 \dfrac{\lambda^*}{\ve^2}\ds \intO \pa_z v_1\, \sigma_{12}^*\, \tau_{11} \leq \lambda^* |\sigma_{12}^*|_\infty\, \lt|\dfrac1\ve \pa_z v_1\rt|\, \lt|\dfrac1\ve \a\rt| \leq \lambda^* |\sigma_{12}^*|_\infty\lt(\frac1{4c_3} \lt|\dfrac1\ve \pa_z v_1\rt|^2+ c_3\lt|\dfrac1\ve \a\rt|^2\rt) .
\end{equation*}
Imposing $\lambda^* |\sigma_{12}^*|_\infty \leq \chi$ is enough to ensure that the coefficients are less than $r(1-r)\nu/6$ and $1/24$.
\end{itemize}

\item Now, multiplying equation \eqref{EQ_rest:e} by $\dfrac{2\b}\ve$ and integrating over $\Om$, with the same reasoning as in the preceding step it follows:
\begin{equation}\label{apriori-b}
\begin{split}
  \dfrac{\lambda^*}{\ve} \dfrac{\di}{\di t} |\tau_{12}|^2 +\dfrac{\lambda^*}\ve &\intO N(\bv,\tau_{11}-\tau_{22})\,\tau_{12} + 2\lt|\dfrac1\ve \b\rt|^2
+\frac{2\eta}\ve|\nabla_\ve \tau_{12}|^2\\
&=\dfrac{2r\nu}\ve \intO \lt(\pa_x v_2+\dfrac1\ve \pa_z v_1\rt)\,\tau_{12}+\dfrac2\ve \intO L_{12}\, \tau_{12}+ \dfrac2{\ve^2} \intO L_{12}' \, \tau_{12}
\end{split}
\end{equation}
The terms in $L_{12}$ and $L_{12}'$ are of the same type as the ones in $L_{11}$ and $L_{11}'$, and are treated very similarly to them, applying Young inequality, and assuming smallness assumptions on $\ve$.
Thus, let us only write the terms needing additional assumptions.
\begin{asparaitem}[$\star$]
 \item $\lambda^*\ds\intO \pa_x v_2\,(\sigma_{11}^*-\sigma_{22}^*)\,\tau_{12} \leq \lambda^*(|\sigma_{11}^*|_\infty+|\sigma_{22}^*|_\infty)\,|\ve \pa_x v_2|\,\lt|\dfrac1\ve \b\rt|$,
and it is enough to assume that $\lambda^*(|\sigma_{11}^*|_\infty+|\sigma_{22}^*|_\infty)\leq\chi$.

\item $\dfrac{2\lambda^*}\ve\ds\intO  v_2\,\pa_z \sigma_{12}^*\,\tau_{12} \leq 2\lambda^*|\pa_z \sigma_{12}^*|_\infty\,|\pa_z v_2|\,\lt|\dfrac1\ve \b\rt|$, and we assume that $2\lambda^*|\pa_z \sigma_{12}^*|_\infty\leq \chi$.

\item $\dfrac{\lambda^*}{\ve^2}\ds\intO \pa_z u^*\, (\tau_{11}-\tau_{22})\,\tau_{12} \leq \lambda^*|\pa_z u_1^*|_\infty \lt(\lt|\dfrac1\ve \a\rt| + \lt|\dfrac1\ve \c\rt| \rt)\, \lt|\dfrac1\ve \b\rt|$, it has already been assumed that $\lambda^*|\pa_z u_1^*|_\infty\leq1/12$.

\item $\dfrac{\lambda^*}{\ve^2}\ds\intO \pa_z v_1\, (\sigma_{11}^*-\sigma_{22}^*)\,\tau_{12} \leq \lambda^*(|\sigma_{11}^*|_\infty +|\sigma_{22}^*|_\infty ) \lt|\dfrac1\ve \pa_z v_1\rt|\,\lt|\dfrac1\ve \b\rt|$, it has already been assumed that $\lambda^*(|\sigma_{11}^*|_\infty+|\sigma_{22}^*|_\infty )\leq \chi$.
\end{asparaitem}

\item Multiplying \eqref{EQ_rest:f} by $\dfrac{\tau_{22}}\ve$, and estimating the terms just as the ones in $\a$, it follows
\begin{equation}\label{apriori-c}
\begin{split} 
 \dfrac{\lambda^*}{2\ve} \dfrac{\di}{\di t} |\tau_{22}|^2 +\dfrac{\lambda^*}\ve \intO N(\bv,\tau_{12})\,\tau_{22}& + \lt|\dfrac1\ve \c\rt|^2+\frac\eta\ve |\nabla_\ve \tau_{22}|^2\\
&=\dfrac{2r\nu}{\ve^2} \intO \pa_z v_2\,c+\dfrac1\ve \intO L_{22}\, \tau_{22} + \dfrac1{\ve^2} \intO L_{22}' \, \tau_{22}.
\end{split}
\end{equation}
Assuming that $\lambda|\sigma_{12}^*|_\infty\leq\chi$, $\lambda^*|\pa_z \sigma_{11}^*|_\infty\leq\chi$ and $\lambda^*|\pa_z u_1^*|_\infty\leq 1/12$, all the terms $\dfrac1\ve \intO L_{22}\, \c$ and $\dfrac1{\ve^2} \intO L_{22}' \, \c$ are bounded and estimated as in Step 1.

\item Summing \eqref{apriori-a}, \eqref{apriori-b} and \eqref{apriori-c}, and noticing that
\begin{equation}
\begin{split}
 -\intO N(\bv,\tau_{12})\,\tau_{11}+&\intO N(\bv, \tau_{11}-\tau_{22})\,\tau_{12}+ \intO N(\bv,\tau_{12})\,\tau_{22} \\
&= \intO \lt(\ve\pa_x v_2-\dfrac1\ve \pa_z v_1\rt)\, \lt(-\b\,\tau_{11} +(\tau_{11}-\tau_{22})\,\tau_{12} + \b\,\c\rt)= 0,
\end{split}\nonumber
\end{equation}
it follows that for $\ve$ small enough
\begin{equation}
\begin{split}
  \dfrac{\lambda^*}{2\ve} \dfrac{\di}{\di t} \lt(|\tau_{11}|^2 + 2|\tau_{12}|^2 +|\tau_{22}|^2\rt)& + \frac12\lt(\lt|\dfrac1\ve \a\rt|^2+ 2\lt|\dfrac1\ve \b\rt|^2+ \lt|\dfrac1\ve \tau_{22} \rt|^2\rt)+ \frac\eta\ve \lt(|\nabla_\ve \tau_{11}|^2 + 2 |\nabla_\ve \tau_{12}|^2 + |\nabla_\ve \tau_{22}|^2\rt) \\
&\leq \mathcal D_1 + \mathcal D_2+r(1-r)\nu^2\lt(|\nabla_\ve v_1|^2+|\ve \nabla_\ve v_2|^2\rt)+C,
\end{split}\nonumber
\end{equation}
where we recognized the terms $\mathcal D_1 + \mathcal D_2$, and where $C$ is a constant independent of $\ve$.
\end{asparaenum}
\end{proof}

\nit From now on, let us come back to the notation with the superscripts $^\ee$, denoting the dependence on $\ve$ and $\eta$.
\begin{theo}\label{th:cv}
Suppose that the solution $\bu^*,\bs^*$ of system \eqref{EQ_star}-\eqref{BC_star} satisfies the following smallness assumptions
\begin{equation}\label{hyp_petit}
\lambda^* |\pa_z u_1^*|_\infty\leq1/12, ~ \lambda^* |\sigma_{12}^*|_\infty\leq\chi, ~ \lambda^*(|\sigma_{11}^*|_\infty+|\sigma_{22}^*|_\infty)\leq \chi, ~ 2\lambda^*|\pa_z \sigma_{12}^*|_\infty \leq \chi,~
\lambda^* |\pa_z \sigma_{11}^*|_\infty \leq \chi, 
\end{equation}
where $\chi = \frac\nu6 \sqrt{r(1-r)}$.
Then the following convergences hold up to subsequences when $\eta$ and then $\ve$ tend to zero:
\begin{gather}
 u_1^\ee\rightarrow u_1^*, ~~ \pa_z u_1^\ee \rightarrow \pa_z u_1^* ,~~ 
\pa_x u_1^\ee \rightharpoonup \pa_x u_1^* \qquad\text{in $L^2(0,T,L^2(\Om))$},\label{cvu}\\
u_2^\ee\rightarrow 0, ~~ \pa_z u_2^\ee \rightarrow 0,~ ~ 
 \pa_x u_2^\ee \rightharpoonup 0 \qquad \text{in $L^2(0,T,L^2(\Om))$},\label{cvv}\\
\ve \bs^\ee \rightarrow \bs^* \qquad\text{in $L^2(0,T,L^2(\Om))$},
\label{cvs}\\
u_1^\ee\rightharpoonup^* u_1^*  , ~~ u_2^\ee \rightharpoonup^* 0  ,~~ 
\ve\bs^\ee \rightarrow \bs^*  \qquad\text{in $L^\infty(0,T,L^2(\Om))$}. \label{cvtps}
\end{gather}
\end{theo}
\begin{proof}
Summing \eqref{apriori-r1r2}, \eqref{apriori-abc} implies that for $\ve$ small enough (i.e. if assumption \eqref{hyp_petit} is satisfied):
\begin{equation}
 \begin{split}
  &r\nu\rho \dfrac{\di}{\di t} \lt(|v_1^\ee|^2+|\ve v_2^\ee|^2\rt) + \dfrac{\lambda^*}{2\ve} \dfrac{\di}{\di t} \lt(|\tau_{11}^\ee|^2+2|\tau_{12}^\ee|^2+|\tau_{22}^\ee|^2\rt) + \frac\eta\ve \lt(|\nabla_\ve \tau_{11}^\ee|^2 + 2 |\nabla_\ve \tau_{12}^\ee|^2 + |\nabla_\ve \tau_{22}^\ee|^2\rt)\\
&+ \frac{r(1-r)\nu^2}2\lt(|\pa_x v_1^\ee|^2+\lt|\dfrac1\ve \pa_z v_1^\ee\rt|^2+ |\ve \pa_x v_2^\ee|^2+ |\pa_z v_2^\ee|^2\rt) + \frac12\lt|\dfrac1\ve \a^\ee\rt|^2+ \lt|\dfrac1\ve \b^\ee\rt|^2+\frac12\lt|\dfrac1\ve \c^\ee\rt|^2 \leq C.
\end{split}
\end{equation}
From this inequality, it follows that $\bv^\ee$ converges to $\bv^\ve$ in $\H^1(\Om)$ and $\bt^\ee$ converges $\bt^\ve$ in $\L^2(\Om)$, as $\eta$ tends to zero. $\bv^\ve$ and $\bt^\ve$ are the solutions solutions of \eqref{EQ_rest} without the terms $\eta \Delta \bt^\ee$.
Indeed, recalling the weak formulation of the system \eqref{EQ_rest}, it suffices to notice that Hölder's inequality allows to treat the term $\eta \Delta \bt^\ee$:
\begin{equation*}
 \eta \intO \nabla_\ve \bt^\ee\cdot\nabla_\ve \bds\phi \leq \eta^{1/2}\Big(\underbrace{\eta|\nabla_\ve \bt^\ee|^2}_{\leq C} + |\nabla_\ve \bds\phi|^2\Big) \xrightarrow[\eta\rightarrow 0]{} 0,\quad \forall \bds\phi \in \H^1_0(\Om).
\end{equation*}

\nit Moreover, $\bv^\ve$ and $\bt^\ve$ satisfy the following estimate:
\begin{equation}\label{est}
 \begin{split}
 & r\nu\rho \dfrac{\di}{\di t} \lt(|v_1^\ve|^2+|\ve v_2^\ve|^2\rt) + \dfrac{\lambda^*}{2\ve} \dfrac{\di}{\di t} \lt(|\tau_{11}^\ve|^2+2|\tau_{12}^\ve|^2+|\tau_{22}^\ve|^2\rt) \\
&+ \frac12 r(1-r)\nu^2\lt(|\pa_x v_1^\ve|^2+\lt|\dfrac1\ve \pa_z v_1^\ve\rt|^2+ |\ve \pa_x v_2^\ve|^2+ |\pa_z v_2^\ve|^2\rt) + \frac12\lt|\dfrac1\ve \a^\ve\rt|^2+ \lt|\dfrac1\ve \b^\ve\rt|^2+\frac12\lt|\dfrac1\ve \c^\ve\rt|^2 \leq C.
\end{split}
\end{equation}

\nit It remains to pass to the limit as $\ve $ tends to zero. After integrating \eqref{est} between 0 and $T$, it yields that
\begin{asparaitem}[$\triangleright$]
\item $\|v_1^\ve\|_{L^2(L^2)}\leq \|\pa_z v_1^\ve\|_{L^2(L^2)}\leq C\ve$, thus the following convergences hold in $L^2(0,T,L^2(\Om))$ 
as $\ve$ tends to zero: 
\begin{equation}\label{cvr1a}
v_1^\ve\rightarrow 0 \quad \text{and} \quad  \pa_z v_1^\ve \rightarrow 0.
\end{equation}
From these convergences, it follows that $u_1^\ve=u_1^*+v_1^\ve\rightarrow u_1^*$ in $L^2(0,T,L^2(\Om))$ and $\pa_z u_1^\ve \rightarrow \pa_z u_1^*$ in $L^2(0,T,L^2(\Om))$.

\item $\|\pa_x v_1^\ve\|_{L^2(L^2)}\leq C$, thus $\pa_x v_1^\ve$ converges weakly in $L^2(0,T,L^2(\Om))$. Now, since it is already known that $u_1^\ve\rightarrow u_1^*$, it follows that $\pa_x u_1^\ve \rightharpoonup \pa_x u_1^*$ in $L^2(0,T,L^2(\Om))$.

\item Similarly $\|v_2^\ve\|_{L^2(L^2)}\leq \|\pa_z v_2^\ve\|_{L^2(L^2)}\leq C$, thus $\ve v_2^\ve$ and $\ve\pa_z v_2^\ve$ converge strongly to zero in $L^2(0,T,L^2(\Om))$, and thus $u_2^\ve=\ve u_2^*+\ve v_2^\ve\rightarrow 0$ in $L^2(0,T,L^2(\Om))$, and $\pa_z u_2^\ve \rightarrow 0$ in $L^2(0,T,L^2(\Om))$.

\item $\|\pa_x v_2^\ve\|_{L^2(L^2)}\leq \dfrac{C}\ve$, thus $\pa_x u_2^\ve$ converges weakly in $L^2(0,T,L^2(\Om))$. Since $u_2^\ve\rightarrow 0$, it implies that $\pa_x u_2^\ve \rightharpoonup 0$ in $L^2(0,T,L^2(\Om))$.

\item $\|\a^\ve\|_{L^2(L^2)},\|\b^\ve\|_{L^2(L^2)},\|\c^\ve\|_{L^2(L^2)}\leq C\ve$, therefore $\a^\ve,\tau_{12}^\ve,\c^\ve \rightarrow 0$ in $L^2(0,T,L^2(\Om))$. Thus $\ve\sigma_{11}^\ve = \sigma_{11}^*+\tau_{11}^\ve \rightarrow \sigma_{11}^*$ in $L^2(0,T,L^2(\Om))$, and in the same way $\ve \sigma_{12}^\ve\rightarrow \sigma_{12}^*$ in $L^2(0,T,L^2(\Om))$, $\ve \sigma_{22}^\ve\rightarrow \sigma_{22}^*$ in $L^2(0,T,L^2(\Om))$.

\item From the terms with the derivatives in time, using the fact that $\bv^\ve|_{t=0}= \bu^\ve_0-\bu^* \in \L^2(\Om)$ and $\bt^\ve|_{t=0}=\bs^\ve_0-\bs^*\in \L^2(\Om)$ are bounded independently of $\ve$, we can conclude that 
\begin{equation*}
 \|\bv^\ve\|_{L^\infty(L^2)} \leq C \qquad \text{and}\qquad \|\bt^\ve\|_{L^\infty(L^2)} \leq C\sqrt\ve.
\end{equation*}
These estimates and the uniqueness of the limit imply that $v_1^\ve$ and $\ve v_2^\ve$ converge weakly-* in $L^\infty(0,T,L^2(\Om))$ toward zero, and that $\bt^\ve$ converges strongly in $L^\infty(0,T,\L^2(\Om))$ toward zero, which proves the last estimate \eqref{cvtps}.
\end{asparaitem}
\end{proof}
\nit Note that in a simplified case (with a simpler geometry), the hypothesis \eqref{hyp_petit} is satisfied under a small data assumption on the physical parameters.
\begin{rem}
When $h$ is constant with respect to $x$, $p^*$ is also independent of $x$, so that equation \eqref{EQ_star} reduces to 
\begin{equation*}
 -(1-r)\pa_z^2 u_1^* -r\dfrac{\pa}{\pa z}\lt(\dfrac{\pa_z u_1^*}{1+\lambda^{*2}|\pa_z u_1^*|^2}\rt)=0.
\end{equation*}
It has been shown in \cite{BCF04} for example that for $r<8/9$ this equation admits a unique solution $u_1^*=s(1-\frac{z}h)$.

\nit Now, it follows that $\sigma_{12}^*=\dfrac{r\nu \pa_z u_1^*}{1+\lambda^{*2}|\pa_z u_1^*|^2}=\dfrac{-r\nu s}{h+\lambda^{*2}s^2/h}$, and $\sigma_{11}^* = -\sigma_{22}^* = -\lambda^* \pa_z u_1^* \sigma_{12}^*=\dfrac{-r\nu s^2 \lambda^*}{h^2+\lambda^{*2}s^2}$.

\nit In this case, hypothesis \eqref{hyp_petit} becomes more simple. Since $\pa_z u_1^*=-s/h$, $\sigma_{11}^*$ and $\sigma_{12}^*$ are constant with respect to $z$, so that the last two conditions are trivially verified. Using the fact that $r<8/9$,
it leads to a smallness condition on $s\lambda^*$ with respect to $h$ ($s\lambda^* \leq h/12$ is enough in order to satisfy all conditions).

\nit Observe that this condition is not optimal, but it shows that in the simplified case when $h(x)$ is constant, a simple choice of the parameters $s$, $\lambda^*$ and $h$ satisfies hypothesis \eqref{hyp_petit}. 
\end{rem}

\subsection{Convergence of the pressure}
\nit It remains to prove the convergence of the pressure.
\begin{theo}\label{mainth_p}
 Under the same smallness assumption \eqref{hyp_petit}, the following convergence result holds up to a subsequence for $p$:
\begin{equation}\label{cvp}
 \ve^2 p \mathop\rightarrow\lm_{\ve\rightarrow 0} p^* \quad \text{in} ~ \mathcal D'(0,T,L^2(\Om)).
\end{equation}
\end{theo}
\begin{proof}
Throughout the proof, $C$ will denote some generic constants independent of $\ve$.
Let $\ve \leq 1$. Let us integrate over $\Om_T = \Om\times (0,T)$ equation \eqref{EQ_rest:a} multiplied by $\ve^2 \vphi_1$, for any function $\phi_1\in H^1_0(\Om)$. It follows:
\begin{equation}\label{weakp}
\begin{split} 
\rho \ve^2 \intOt &\pa_t v_1\phi_1 +\rho\ve^2\intOt v_1 \pa_x v_1 \phi_1+ \rho \ve\intOt v_2 \pa_z v_1 \phi_1 + (1-r)\nu \ve^2 \intOt\pa_x v_1 \,\pa_x \phi_1 + (1-r)\nu \intOt\pa_z v_1\, \pa_z \phi_1\\
&+ \intOt \pa_x q\, \phi_1 = -\ve \intOt \a\pa_x \phi_1 -\intOt \b\pa_z \phi_1 + \ve^2\intOt L_1\phi_1 + \ve\intOt C_1\phi_1, \quad \forall \phi_1\in H^1_0(\Om).
\end{split}
\end{equation}
Using the fact that $\phi_1$ is independent of $t$, the first term becomes
\begin{equation*}
 \rho \ve^2 \intOt \pa_t v_1\phi_1 =  \rho \ve^2 \intO \phi_1 \int\lm_0^T\pa_t v_1 = \rho \ve^2 \intO \phi_1 (v_1(T) -v_1(0)),
\end{equation*}
where $v_1(0) = u_{1_{\scriptstyle 0}}-u_1^*$ denotes the value of $v_1$ at time $t=0$. Now, introducing 
\begin{equation*}
 \pi = \int\lm_0^T q \,\di t,
\end{equation*}
and using integration by parts for the pressure term (the boundary term is zero since $\phi_1\in H^1_0(\Om)$), \eqref{weakp} becomes: $\forall \phi_1\in H^1_0(\Om)$,
\begin{equation}
\begin{split} 
\rho \ve^2 &\intO \phi_1(v_1(T) -v_1(0)) +\rho\ve^2\intOt v_1 \pa_x v_1 \phi_1+ \rho \ve\intOt r_2 \pa_z v_1 \phi_1 + (1-r)\nu \ve^2 \intOt\pa_x v_1 \,\pa_x \phi_1 \\&+ (1-r)\nu \intOt\pa_z v_1\, \pa_z \phi_1
- \intO \pi\, \pa_x\phi_1 = -\ve \intOt \tau_{11}\pa_x \phi_1 -\intOt \tau_{12}\pa_z \phi_1 + \ve^2\intOt L_1\phi_1 + \ve\intOt C_1\phi_1.
\end{split}\nonumber
\end{equation}

\nit It remains to estimate all terms independent of $\pi$. 
The non-linear terms are to bee handled with care, since $\phi_1 \notin L^\infty(\Om)$. Proceeding as in \cite{BCC99}, Hölder inequality with exponents $2+\delta$, $\delta'$ and 2 leads:
\begin{equation}\label{estnlin_p1}
 \lt|\intOt v_1 \pa_x v_1 \phi_1 \rt| \leq  |\phi_1|_{\delta'} \int\lm_0^T |v_1|_{2+\delta}\,| \pa_x v_1 |,
\end{equation}
where $\dfrac1{2+\delta} + \dfrac12 + \dfrac1{\delta'} =1$ (which implies that $\delta'=\dfrac{2(2+\delta)}\delta$). According to interpolation theory, $\lt[L^2,L^4\rt]_\theta = L^{2+\delta}$ for $\theta = \dfrac\delta{2+\delta}$, and the following estimate holds:
\begin{equation*}
 |v_1|_{2+\delta} \leq C |v_1|_4^\theta \, |v_1|^{1-\theta}.
\end{equation*}
Moreover Lemma 3.2 of \cite{ABC94} states that for $v_1\in H^1_0(\Om)$, it holds:
\begin{equation*}
 |v_1|_4 \leq \sqrt2 |\pa_x v_1|^{1/4}\, |\pa_z v_1|^{3/4}.
\end{equation*}
Using the two last inequalities and Poincaré inequality, \eqref{estnlin_p1} becomes
\begin{equation*}
  \rho\ve^2 \lt|\ \intOt v_1 \pa_x v_1 \phi_1 \rt| \leq \rho \ve^2 |\phi_1|_{\delta'} C \int\lm_0^T   |\pa_x v_1|^{\theta/4} \, |\pa_z v_1|^{3\theta/4} |\pa_z v_1|^{1-\theta} |\pa_x v_1|,
\end{equation*}
and Hölder inequality implies that
\begin{equation*}
  \rho\ve^2\intOt v_1 \pa_x v_1 \phi_1 \leq \rho \ve^2 |\phi_1|_{\delta'} C \|\pa_x v_1\|_{L^2(\Om_T)}^{1+\theta/4} \, \|\pa_z v_1\|_{L^2(\Om_T)}^{1-\theta/4} .
\end{equation*}
Now, choose $\theta$ (and thus $\delta$) such that $\delta' \geq 6$. It suffices to take $\theta \leq \frac13$, for example take $\theta=\dfrac13$. Then $\delta'=6$, and the usual Sobolev embeddings read $H^1(\Om) \hookrightarrow L^6(\Om)$ (which is true in dimension 2 or 3).
Therefore, the last estimate becomes
\begin{equation*}
  \rho\ve^2\intOt v_1 \pa_x v_1 \phi_1 \leq \rho \ve^2 C\|\phi_1\|_{H^1} \|\pa_x v_1\|_{L^2(\Om_T)}^{13/12} \, \|\pa_z v_1\|_{L^2(\Om_T)}^{11/12} .
\end{equation*}
Now, recalling that $\|\pa_z v_1\|_{L^2(L^2)}\leq C\ve$ and $\|\pa_x v_1\|_{L^2(L^2)}\leq C$, we conclude
\begin{equation*}
  \rho\ve^2\intOt v_1 \pa_x v_1 \phi_1 \leq \rho \ve^2 C\|\phi_1\|_{H^1}  \ve^{11/12} =  \rho \ve^{2+11/12} C\|\phi_1\|_{H^1} \leq \ C \ve \|\phi_1\|_{H^1}.
\end{equation*}
In a similar way, it holds
\begin{equation*}
  \rho\ve\intOt r_2 \pa_z v_1 \phi_1 \leq  \rho \ve^{2-1/12} C\|\phi_1\|_{H^1} \leq \widetilde C \ve \|\phi_1\|_{H^1}.
\end{equation*}
For the term $\rho \ve^2 \intO \phi_1(v_1(T) -v_1(0))$, we apply Cauchy-Schwarz inequality. $v_1(0)$ is bounded, and for $v_1(T)$, we use Poincaré inequality. It follows, using the fact that $|\pa_z v_1|\leq C\ve$:
\begin{equation*}
 \rho \ve^2 \intO \phi_1(v_1(T) -v_1(0)) \leq (C|v_1| + C) \ve^2 \|\phi_1\|_{H^1}  \leq (C|\pa_z v_1| + C) \ve^2 \|\phi_1\|_{H^1} \leq C\ve^2 \|\phi_1\|_{H^1} \leq C\ve \|\phi_1\|_{H^1}.
\end{equation*}

\nit For the other linear terms, a simple application of Cauchy-Schwarz inequality allows to obtain similar estimates. 
Indeed, it suffices to use the estimate \eqref{est} in order to estimate the $L^2$-norm of $\pa_x v_1$, $\pa_z v_1$, $\a$, $\b$, $L_1$, $C_1$. For example, since $|\pa_x v_1| \leq C$, the following estimate holds:
\begin{equation*}
 \rho \ve^2\intO \pa_x v_1 \, \pa_x \phi_1 \leq \rho \ve^2 |\pa_x v_1|\, |\pa_x \phi_1| \leq C \ve^2 \|\phi_1\|_{H^1}. 
\end{equation*}
For the terms $L_1$ and $C_1$, $C_1$ and the constant part of $L_1$ are obviously bounded uniformly in $\ve$. It remains to estimate the linear term $\mathcal L_1$ of $L_1$. Recalling its definition and using Poincaré inequality in the second estimate:
\begin{equation*}
 |\mathcal L_1| \leq C\lt(|v_1|+|v_2|+|\pa_x v_1|+|\pa_z v_1|\rt) \leq C\lt(|\pa_z v_1| + |\pa_x v_1| + |\pa_z v_2|\rt).
\end{equation*}
Using again \eqref{est}, the boundedness of $\mathcal L_1$ follows:
\begin{equation*}
 |\mathcal L_1| \leq C.
\end{equation*}
Hence $\forall \phi_1\in H^1_0(\Om)$:
\begin{equation*}
 \intO\pa_x \pi\, \phi_1\leq C\lt(\ve +\ve^2|\pa_x v_1| +|\pa_z v_1|+\ve|\tau_{11}|+|\tau_{12}|+\ve^2|L_1|+\ve|C_1|\rt)\|\phi_1\|_{H^1} \leq C\ve\|\phi_1\|_{H^1}.
\end{equation*}
The same approach with \eqref{EQ_rest:b} gives a similar estimate, for all $\phi_2\in H^1_0(\Om)$:
\begin{equation*}
 \intO\pa_z \pi \,\vphi_2\leq C\lt(\ve +\ve^4|\pa_x v_2|+\ve^2|\pa_z v_2|+\ve^2|\tau_{12}|+\ve|\tau_{22}|+\ve^2|L_2|+\ve|C_2|\rt) \|\phi_2\|_{H^1}\leq C\ve\|\phi_2\|_{H^1}.
\end{equation*}
Thus we can conclude that $\|\nabla \pi\|_{L^\infty(H^{-1})} \leq C\ve$. 

\nit Now recall that for $f\in L^2_0(\Om)$, it holds that $|f|\leq \|\nabla q\|_{H^{-1}}$ (see for example \cite{Tem79}).
Since $p\in L^2_0(\Om)$ and $p^*\in L^2_0(\Om)$, $q$ lies in $L^2_0(\Om)$. From the definition of $\pi$ as function of $q$, it is clear that $\pi\in L^2_0(\Om)$.

\nit This allows to deduce
\begin{equation*}
 |\pi|_{L^\infty(L^2)}\leq \|\nabla \pi\|_{L^\infty(H^{-1})} \leq C\ve \rightarrow 0,
\end{equation*}
thus $\pi$ tends to zero in $L^\infty(0,T,L^2_0(\Om))$ when $\ve\rightarrow 0$. Now, since $q = \dfrac{\pa\pi}{\pa t}$, it follows that $q$ tends to zero in $\mathcal D'(0,T,L^2_0(\Om))$, and therefore:
\begin{equation*}
 \ve^2 p \mathop\rightarrow\lm_{\ve\rightarrow 0} p^* \quad \text{in} ~ \mathcal D'(0,T,L^2(\Om)).
\end{equation*}
This finishes the proof.
\end{proof}

\subsection{Open problems}
This work concerns only the solutions of the problem \eqref{EQ} that are obtained as the limit of the regularized problem we chose (with an additional term $-\eta \Delta \bs$). Since there is no uniqueness result for problem \eqref{EQ}, it is not known how other solutions behave.
\smallskip
\par  \nit Formally, the passing to the limit can be done for $a\neq 0$ (see \cite{BCM07}), and a similar limit problem (involving the parameter $a$, but of the same structure). However, the proof of the existence theorem in $\hat \Om^\ve$ strongly relies on the fact that $a=0$. No global results are proved in the case $a\neq 0$. 
\smallskip
\par \nit Last, since the computations are independent of the dimension of the domain $\Om$, the result should be true in the three-dimensional case. The limit problem on $(\bu^*,p^*,\bs^*)$ reads:
\begin{System}\label{EQ_star3D}
(1-r)\nu \pa_z^2 u_1^*-\pa_x p^* +\pa_z \sigma_{13}^*=0,\\
(1-r)\nu \pa_z^2 u_2^*-\pa_x p^* +\pa_z \sigma_{23}^*=0,\\
\pa_z p^*=0,\\
\nabla \cdot \bu^* =0,\\
-\lambda^* \pa_z u_1^* \sigma_{13}^* + \sigma_{11}^*=0,\\
-\dfrac{\lambda^*}2 \pa_z u_1^*\sigma_{13}^* -\pa_z u_2^* \sigma_{23}^* + \sigma_{12}^* =0,\\
-\lambda^* \pa_z u_2^* \sigma_{23}^* + \sigma_{22}^*=0,\\
\dfrac{\lambda^*}2 \pa_z u_2^*(\sigma_{33}^* -\sigma_{22}^*)- \dfrac{\lambda^*}2 \pa_z u_1^*\sigma_{12}^*+ \sigma_{23}^* =r\nu \pa_z u_2^*,\\
\lambda^* \lt(\pa_z u_1^* \sigma_{13}^*+ \pa_z u_2^* \sigma_{23}^*\rt) + \sigma_{33}^*=0,\\
\dfrac{\lambda^*}2 \pa_z u_1^*(\sigma_{33}^* -\sigma_{11}^*) - \dfrac{\lambda^*}2 \pa_z u_2^*\sigma_{12}^*+ \sigma_{12}^* =r\nu \pa_z u_1^*.
\end{System}

\section*{Acknowledgements}
The authors would like to thank P. Mironescu for useful discussions concerning section \ref{s:reg} of this paper. 

\bibliographystyle{plain}

\end{document}